\documentclass[psamsfonts,12pt]{amsart}
\usepackage{graphicx}  
\usepackage{amsmath,amssymb,amsthm, amscd}
\usepackage{amssymb,fge}
\usepackage{tikz-cd}

\usepackage{color}

\DeclareMathOperator{\car}{char}

\DeclareMathOperator{\Pic}{Pic}
\DeclareMathOperator{\Orb}{Orb}

\DeclareMathOperator{\CritVal}{CritVal}
\DeclareMathOperator{\disc}{disc}

\DeclareMathOperator{\gen}{genus}

\DeclareMathOperator{\PrePer}{PrePer}

\newcommand{\mysetminus}{\mathbin{\fgebackslash}}

\newenvironment{Res}{{\rm Res\,}}{}

\newtheorem{theorem}{Theorem}[section]
\newtheorem{lemma}[theorem]{Lemma}
\newtheorem{proposition}[theorem]{Proposition}
\newtheorem{proposition-definition}[theorem]{Proposition-Definition}
\newtheorem{corollary}[theorem]{Corollary}
\newtheorem{conjecture}[theorem]{Conjecture}

\newtheorem{question}[theorem]{Question} 

\theoremstyle{definition}
\newtheorem*{definition}{Definition}
\theoremstyle{remark}
\newtheorem{remark}[theorem]{Remark}
\newtheorem{example}[theorem]{Example}

\usepackage[margin=1in]{geometry}  
\usepackage{graphicx}              
\usepackage{amsmath}               
\usepackage{amsfonts}              
\usepackage{amsthm}                
\usepackage{verbatim}              
\usepackage{indentfirst}           
\usepackage{underscore}
\usepackage{enumitem}
\usepackage{relsize}
\usepackage{varwidth}
\usepackage{mathtools} 
\usepackage{hyperref}
\hypersetup{
    colorlinks=true,
    linkcolor=blue,
    filecolor=magenta,      
    urlcolor=cyan,
}
\usepackage{amssymb}
\usepackage[all]{xy}
\makeatletter
\renewcommand*\env@matrix[1][*\c@MaxMatrixCols c]{%
  \hskip -\arraycolsep
  \let\@ifnextchar\new@ifnextchar
  \array{#1}}
\makeatother

\theoremstyle{remark}

\newcommand*{\Scale}[2][4]{\scalebox{#1}{$#2$}}%

\newcommand{\Mod}[1]{\ (\textup{mod}\ #1)}     

\newcommand{\sep}{\operatorname{sep}} 

\title[Stochastic Canonical Heights]{Stochastic Canonical Heights} 
\author[Vivian O. Healey]{Vivian Olsiewski Healey}
\address{Department of Mathematics, University of Chicago, 5734 S. University Ave., Chicago, IL 60637.}
\email{vhealey@uchicago.edu} 
\author[Wade Hindes]{Wade Hindes}
\address{Department of Mathematics, The Graduate Center, City University of New York (CUNY); 365 Fifth Avenue, New York, NY 10016, USA.}
\email{whindes@gc.cuny.edu}
\begin{document}
\maketitle
\renewcommand{\thefootnote}{}
\footnote{The first author was partially supported by NSF grant DMS-1246999.}
\begin{abstract} We construct height functions defined stochastically on projective varieties equipped with endomorphisms, and we prove that these functions satisfy analogs of the usual properties of canonical heights. Moreover, we give a dynamical interpretation of the kernel of these stochastic height functions, and in the case of the projective line, we relate the size of this kernel to the Julia sets of the original maps. Finally, as an application, we establish the finiteness of some generalized Zsigmondy sets over global fields.    
\end{abstract} 
\section{Introduction} 
The canonical height \cite{Call-Silverman} associated to a smooth projective variety $V$ equipped with an endomorphism $\phi:V\rightarrow V$ is an indispensable tool for studying the arithmetic properties of the corresponding discrete dynamical system $(V,\phi)$. However, many varieties of interest in number theory (e.g. projective space) possess many such maps, and given that composition is not commutative in general, the dynamical systems generated by a set of maps can differ greatly from the dynamics of a single map. In this paper, we address this problem and construct a new height function that measures the collective action of a set of maps on a fixed variety. 

To accurately characterize the dynamics of a set of maps, one must encode ``how often" to expect a particular map to appear at any stage of composition - a slight alteration in the likelihood of applying a particular map can drastically change the overall dynamics. Moreover, there is no intrinsic reason why all maps should be given equal weight. We therefore use the language and tools of probability in our constructions. In so doing, we are in essence analyzing a sort of random walk on the variety (formally a Markov chain), where the stochastic motion is generated by random evaluation of maps in some fixed set.   
\begin{example}{\label{eg:twomapsfaircoin}} To make this clear, the reader is encouraged to keep in mind the \emph{two maps and a fair coin} example: suppose that we have two maps $S=\{\phi_1,\phi_2\}$ on a variety $\phi_i: V\rightarrow V$. Then for any given point $P\in V$, we can flip a coin to determine whether to evaluate $\phi_1$ or $\phi_2$ at $P$, assigning each outcome an equal probability of $1/2$. Now repeat this process for the new point $\phi_1(P)$ or $\phi_2(P)$ and continue inductively in this way, associating to an infinite sequence of coin flips (a type) of orbit of $P$.        
\end{example}       
If $V$ is equipped with a height function, then we can ask how the height of $P$ grows as we move along a particular path. Even more broadly, we can ask about the height growth distribution as we vary over all possible paths (each path weighted by its probability). To make this idea precise, we fix some notation.  Let $K$ be a global field, let $V_{/K}$ be a smooth projective variety over $K$, and let $S$ be a (finite or infinite) set of endomorphisms on $V$ defined over $K$. To define the orbits we will consider, let 
\[\text{$\Phi_{S,n}=\prod_{i=1}^nS$\;\; and\;\; $\Phi_S=\prod_{i=1}^\infty S$}\] be the set of $n$-term (and infinite) sequences of elements of $S$ respectively. Given an infinite sequence $\gamma=(\theta_n)_{n\geq1}\in \Phi_S$ and a positive integer $m\geq1$, we let $\gamma_m=(\theta_i)_{i=1}^m\in\Phi_{S,m}$ and define an action of $\gamma_m$ on $V$ by 
\[\text{$\gamma_m\cdot P=(\theta_m\circ\theta_{m-1}\circ\dots \circ\theta_1)(P)$\hspace{.5cm}for $P\in V$.}\] 
In this way, we define the \emph{orbit} of a point $P\in V$ with respect to a sequence $\gamma\in\Phi_S$ to be: 
\[\Orb_\gamma(P)=\{\gamma_m\cdot P: m\geq1\}.\] 
Finally, if $\nu_1$ is a probability measure on $S$, then we define a probability measure $\nu_m$ on $\Phi_{S,m}$ by the product
\[ \text{$\nu_m(\gamma_m)=\prod_{i=1}^m\nu_1(\theta_i)$,\hspace{.5cm} for $\gamma_m=(\theta_i)_{i=1}^m$.} \] 
That is, each $\gamma_m\in\Phi_{S,m}$ is a sequence of $m$ elements of $S$, and each component of $\gamma_m$ is chosen independently according to $\nu_1$. Likewise, $\nu_1$ induces a probability measure $\nu$ on  the set of infinite sequences $\Phi_S$; see \cite[Theorem 10.4]{ProbabilityText}. We call $(\Phi_S,\mathcal{F},\nu)$ the \emph{probability space of i.i.d sequences} of elements of $S$ distributed according to $\nu_1$; here $\mathcal{F}$ is the $\sigma$-algebra of $\nu$-measurable subsets of $\Phi_S$. 
       
Now for a brief discussion of the relevant material on canonical heights. Let $\eta\in\Pic(V)\otimes\mathbb{R}$ be a divisor class and let $h_{V,\eta}: V(\overline{K})\rightarrow\mathbb{R}$ be a corresponding Weil height function; see, for instance, \cite[\S2-\S4]{Lang}. To define the canonical height for a fixed map $\phi: V\rightarrow V$, one requires that $\eta$ is an eigenclass for $\phi$; the key point in this case is that 
\begin{equation}\label{functoriality}
h_{V,\eta}\circ\phi=\alpha_\phi h_{V,\eta}+O_{V,\eta,\phi}(1)
\end{equation}  
for some $\alpha_\phi\in\mathbb{R}$. With this in mind, we let \[C(V,\eta,\phi):=\sup_{P \in V} \Big\vert h_{V,\eta}(\phi(P))-\alpha_\phi h_{V,\eta}(P)\Big\vert\]
be the smallest constant needed for the bound in (\ref{functoriality}). Then, in order to generalize the construction of canonical heights for a single map to a collection of maps (equivalently, from constant sequences to arbitrary sequences), we define the following fundamental notion.  
\begin{definition} 
A set of endomorphisms $S$ on a projective variety $V$ is \textbf{\emph{height controlled}} with respect to a divisor class $\eta\in\Pic(V)\otimes\mathbb{R}$ if: 
\begin{enumerate}[topsep=8pt, partopsep=5pt, itemsep=7pt]  
\item[\textup{(1)}] For all $\phi\in S$, there exists $\alpha_\phi$ such that: $\phi^*(\eta)=\alpha_\phi\eta$\, and\, $\displaystyle{\inf_{\phi\in S}\,\alpha_\phi>1}$. 
\item[\textup{(2)}] The corresponding height constants are bounded: $\displaystyle{\sup_{\phi\in S}\,C(V,\eta,\phi)}$ is finite.\
\end{enumerate}
\end{definition} 
These properties are easily satisfied for any projective space and any finite set $S$; however, see Example \ref{unicrit} and Remark \ref{htcont} for instances of infinite $S$. We now state our main construction. In what follows, for a finite sequence $\gamma_m=(\theta_i)_{i=1}^m\in\Phi_{S,m}$, we define the degree $\deg_\eta(\gamma_m)=\prod_{i=1}^m \alpha_{\theta_i}$, in what we hope is a pardonable (and instructive) abuse of notation. \vspace{.1cm}       
\begin{theorem}\label{thm:htsiid} Let $K$ be a global field, let $V_{/K}$ be a smooth projective variety over $K$, and let $S$ be a collection of endomorphisms on $V$ equipped with the following:  
\begin{enumerate}[topsep=5pt, partopsep=5pt, itemsep=7pt] 
\item[\textup{(1)}] A common eigendivisor class $\eta\in \Pic(V)\otimes\mathbb{R}$ such that $S$ is height controlled with respect to $\eta$. 
\item[\textup{(2)}] A probability measure $\nu_1$ on $S$. 
\end{enumerate} 
Let $(\Phi_S,\mathcal{F},\nu)$ be the probability space of i.i.d sequences of elements of $S$ distributed according to $\nu_1$, and let $h_{V,\eta}$ be a Weil height function corresponding to $\eta$. Then for all infinite sequences $\gamma\in\Phi_S$ and all points $P\in V(\overline{K})$, the \emph{canonical height},  \vspace{.1cm}       
\[\hat{h}_{V,\eta,P}(\gamma):=\lim_{n\rightarrow\infty}\frac{h_{V,\eta}(\gamma_n\cdot P)}{\deg_\eta(\gamma_n)},\vspace{.1cm}\]
converges. Likewise, the \emph{expected canonical height at $P$} for a random $\gamma$,  \vspace{.1cm} 
\[\mathbb{E}_{\nu}\big[\hat{h}_{V,\eta}\big](P):=\int_{\Phi_S}\hat{h}_{V,\eta,P}(\gamma)\,d\nu,\vspace{.05cm}\]   
converges. Let $d_{\nu,\eta}$ be the (deterministic) constant given by 
\[d_{\nu,\eta}:=\bigg(\sum_{\phi\in S}\frac{\nu_1(\phi)}{\deg_{\eta}(\phi)}\bigg)^{\hspace{-.1cm}-1},\] and let $\nu_k^*$ be the new probability measure on $\Phi_{S,k}$ induced by the pair $(\nu_1,\eta)$ and given by 
\[\nu_k^*(\gamma_k):=\frac{\nu_k(\gamma_k)}{\deg_\eta(\gamma_k)}(d_{\nu,\eta})^{k}.\]  
Then the function $\mathbb{E}_\nu\big[\hat{h}_{V,\eta}\big]: V(\overline{K})\rightarrow\mathbb{R}$ satisfies the following properties: 
\begin{enumerate}[topsep=6pt, partopsep=6pt, itemsep=8pt] 
\item[\textup{(a)}] $\mathbb{E}_\nu\big[\hat{h}_{V,\eta}\big]= h_{V,\eta}+O(1)$. 
\item[\textup{(b)}] $\mathbb{E}_{\nu_k^*}\Big[\mathbb{E}_{\nu}\big[\hat{h}_{V,\eta}\big](\gamma_k\cdot P)\Big]=(d_{\nu,\eta})^{k}\;\mathbb{E}_{\nu}\big[\hat{h}_{V,\eta}\big](P)$ for all $k\geq1$ and all $P\in V(\overline{K})$.\vspace{.12cm} 
\end{enumerate}
\end{theorem} 
\begin{remark} The canonical heights $\hat{h}_{V,\eta,P}(\gamma)$, also denoted $\hat{h}_{V,\eta,\gamma}(P)$ depending on whether we vary the path $\gamma\in\Phi_S$ or the basepoint $P\in V$, were also studied in \cite{Kawaguchi}, and Theorem \ref{thm:htsiid} can be viewed as a generalization of \cite[Proposition C]{Kawaguchi} in two ways: we allow the generating set of functions $S$ to be infinite (under suitable conditions), and we allow arbitrary probability measures on $S$.    
\end{remark}
There are several reasons why we believe that the expected canonical height $\mathbb{E}_\nu\big[\hat{h}_{V,\eta}\big]$ is the right height function to study the collective dynamics of the maps in $S$ (by analogy with the standard canonical height of Call-Silverman). The first reason is that $\mathbb{E}_\nu\big[\hat{h}_{V,\eta}\big]=\hat{h}_{V,\eta,\phi}$ whenever $S=\{\phi\}$ is a singleton with trivial probability measure. The second reason is that $\mathbb{E}_\nu\big[\hat{h}_{V,\eta}\big]$ satisfies a transformation law of similar shape to that of the standard canonical height. The third reason is that $\mathbb{E}_\nu\big[\hat{h}_{V,\eta}\big]$ detects finite $S$-invariant subsets of $V$, an analog of preperiodic points of a fixed map. In what follows, we say that a subset $F\subset V$ is $S$-stable if $\phi(F)\subseteq F$ for all $\phi\in S$. Moreover, we say that the probability measure $\nu_1$ is \emph{strictly positive} if $\nu_1(\phi)>0$ for all $\phi\in S$.      
\begin{corollary}{\label{cor:hteq0}} Let $(V,\eta, S,\nu_1)$ satisfy the conditions of Theorem \ref{thm:htsiid}. If $\eta$ is an ample divisor and $\nu_1$ is strictly positive, then for all $P\in V(\overline{K})$ the following are equivalent: 
\begin{enumerate}[topsep=8pt, partopsep=8pt, itemsep=8pt] 
\item[\textup{(1)}] There is a finite, $S$-stable subset $F_P\subset V$ containing $P$. 
\item[\textup{(2)}] $\nu\big(\{\gamma\in\Phi_S:\, \Orb_\gamma(P)\; \text{is finite}\;\}\big)=1$.  
\item[\textup{(3)}] $\mathbb{E}_\nu[\hat{h}_{V,\eta}](P)=0$. 
\end{enumerate} 
\end{corollary}
With this characterization of the kernel of $\mathbb{E}_\nu[\hat{h}_{V,\eta}]$ in place, one expects that few points (perhaps at most finitely many) have expected height zero unless the maps in $S$ are dynamically dependent in some way. As an example of this heuristic, we note the following application of \cite[Theorem 1.2]{Baker-DeMarco} and Corollary \ref{cor:hteq0} for the projective line. In what follows, given a rational function $\phi\in\overline{\mathbb{Q}}(z)$, we let $\PrePer(\phi)$ denote the set of preperiodic points of $\phi$ in $\mathbb{P}^1(\overline{\mathbb{Q}})$. 
\begin{corollary}\label{cor:julia} Let $S$ be a collection of height controlled rational maps on $\mathbb{P}^1$ defined over a number field $K$, let $\nu_1$ be a strictly positive probability measure on $S$, and let $h$ be the absolute Weil height function on $\mathbb{P}^1(\overline{\mathbb{Q}})$. Then the following are equivalent: 
\begin{enumerate}[topsep=8pt, partopsep=8pt, itemsep=8pt] 
\item[\textup{(1)}] $\big\{P\in\mathbb{P}^1(\overline{\mathbb{Q}})\,:\,\mathbb{E}_\nu[\hat{h}](P)=0\big\}$ is infinite. 
\item[\textup{(2)}] $\displaystyle{\bigcap_{\phi\in S}\PrePer(\phi)}$ is infinite.     
\item[\textup{(3)}] $\PrePer(\phi)=\PrePer(\psi)$ for all $\phi, \psi\in S$. 
\end{enumerate}  
In particular, if the set of points of expected height zero is infinite, then the Julia sets of all of the maps in $S$ coincide.   
\end{corollary}
\begin{remark} For example, if $c_1, c_2, \dots,c_n\in\overline{\mathbb{Q}}$ are distinct algebraic numbers and $S$ is the finite set of quadratic polynomials, 
\[S=\{x^2+c_i: 1\leq i\leq n\},\] equipped with any any strictly positive probability measure, then \cite[\S.4]{Julia:quad} and Corollary \ref{cor:julia} imply 
that \[\big\{P\in\mathbb{P}^1(\overline{\mathbb{Q}})\,:\,\mathbb{E}_\nu[\hat{h}](P)=0\big\}\]
is finite. On the other hand, it is tempting to guess that there are in fact no points of expected height zero for the sets $S$ above. However, this is not true in general, as the example: $S=\{x^2,x^2-1\}$ and the point $P=-1$, shows. Therefore, given a generating set $S$ whose elements do not all share the same Julia set, it is perhaps an interesting problem to determine the finitely many points (over $\overline{\mathbb{Q}}$) of expected height zero. Keep in mind that this is \emph{not} the same problem as studying the intersection of the preperiodic points of the maps in $S$, as the example: $S=\{x^2-2,x^2+x\}$ and $P=0$, shows.      
\end{remark}  
Moreover, when $V=\mathbb{P}^1$ and $S$ is finite, we show that the canonical and expected canonical heights defined in Theorem \ref{thm:htsiid} admit a decomposition into a sum of local heights:  \vspace{.05cm }
\begin{equation}{\label{intro:local}}
\,\hat{h}_\gamma=\frac{1}{\deg(E)}\,\sum_{v\in M_K}n_v\,\hat{\lambda}_{v,\tilde{\gamma}, E}\;\;\;\;\;\;\text{and}\;\;\;\;\;\; \mathbb{E}_\nu\big[\hat{h}\big]=\frac{1}{\deg(E)}\sum_{v\in M_K}n_v\,\mathbb{E}_\nu\big[\hat{\lambda}_{v,\tilde{\gamma},E}\big];  \vspace{.05cm} 
\end{equation} 
see Theorem \ref{thm:localhts} and Corollary \ref{cor:localhts} in Section \ref{sec:local}. We note that similar local Green functions to those we use in Section \ref{sec:local} seem to have first appeared in \cite[\S 6]{Kawaguchi}, although we were unaware of \cite{Kawaguchi} until after the completion of Section \ref{sec:local}. On the other hand, the local heights $\hat{\lambda}_{v,\tilde{\gamma}, E}$ and the decompositions in (\ref{intro:local}) do not appear in \cite{Kawaguchi}. Furthermore, since we use $\hat{\lambda}_{v,\tilde{\gamma}, E}$ and (\ref{intro:local}) to pose some new questions in number theory from a probabilistic point of view, we carry out our construction of Green functions, instead of just citing \cite[\S 6]{Kawaguchi}. 

Finally, as an application of canonical heights, we study the Zsigmondy sets of non-deterministic dynamical systems. More precisely, let $K$ be a global field, let $V=\mathbb{P}^1$, and let $S$ be a finite set of endomorphisms of $V$. Given a sequence $\gamma\in \Phi_S$ and a basepoint $P\in\mathbb{P}^1(K)$, we say that a valuation $v$ of $K$ is a \emph{primitive prime divisor} of $\gamma_n\cdot P$ if:  
\begin{enumerate}[topsep=5pt, partopsep=5pt, itemsep=5pt] 
\item[\textup{(a)}] $v(\gamma_n\cdot P)>0$,
\item[\textup{(b)}] $v(\gamma_m\cdot P)\leq0$ for all $1\leq m\leq n-1$. 
\end{enumerate}
As with deterministic dynamical systems (see \cite{Xander,Tucker,Riccati,Ingram-Silv,Holly}), one expects primitive prime divisors to exist in orbits unless the maps in $S$ are special in some way. To test this heuristic, we define the \emph{Zsigmondy set} of a pair $(\gamma,P)\in\Phi_S\times\mathbb{P}^1(K)$ to be: 
\[ \mathcal{Z}(\gamma,P):=\{n:\, \gamma_n\cdot P\;\text{has no primitive prime divisor}\}.\]
In particular, when $K$ is a number field we use the $abc$-conjecture and ideas from \cite{Tucker} and \cite{AvgZig} to bound the size of the elements of $\mathcal{Z}(\gamma,P)$ under certain mild conditions on the maps in $S$ and the basepoints $P\in\mathbb{P}^1(K)$. Specifically, we restrict our attention to (good) pairs $(\gamma,P)$ in the following sense:  \vspace{.1cm} 
\[G_S:=\Big\{(\gamma,P)\in\Phi_S\times\mathbb{P}^1(K):\,\hat{h}_\gamma(P)>0\;\text{and}\; 0,\infty\not\in \Orb_{\gamma}(P)\cup\Orb_\gamma(0)\Big\}.\vspace{.05cm}\]
\begin{remark} Clearly, $\mathcal{Z}(\gamma,P)$ is infinite whenever $\Orb_\gamma(P)$ is finite (i.e., $\hat{h}_\gamma(P)=0)$; see Remark \ref{rem:ht}. Moreover, certain technicalities arise in our argument when the orbits we consider contain $0$ or $\infty$. However, it is possible that these latter conditions can be relaxed. 
\end{remark} 
In what follows, $\CritVal(S)=\bigcup_{\phi\in S}\CritVal(\phi)$ denotes the union of all critical values of the maps $\phi\in S$; see \cite[p.353]{SilvDyn}. Moreover, 
\[\PrePer(\Phi_S):=\big\{Q\in\mathbb{P}^1(\overline{K})\,: \hat{h}_\psi(Q)=0\;\, \text{for some}\;\, \psi\in\Phi_S\big\}\] 
denotes the set of points with finite orbit for some sequence in $\Phi_S$ (a set of bounded height).      
\begin{theorem}\label{thm:primdiv} Let $K$ be a number field and suppose that $S=\{\phi_1,\phi_2,\dots, \phi_s\}$ is a set of rational maps on $\mathbb{P}^1$ of degree at least $2$, defined over $K$, and satisfying: 
\begin{enumerate}[topsep=6pt, partopsep=6pt, itemsep=7pt] 
\item[\textup{(1)}] $\#\phi_i^{-1}(0)\geq4$,
\item[\textup{(2)}] $0\not\in\CritVal(S)$. 
\end{enumerate}
Then the $abc$-conjecture (\ref{conj:abc}) implies that $\mathcal{Z}(\gamma,P)$ is finite for all $(\gamma,P)\in G_S$.          
\end{theorem} 
We also examine the case of global function fields of prime characteristic. However, since our methods are different in this context, the conditions we impose are also different. In particular, we are confined to polynomial dynamics. On the other hand, the result we obtain is unconditional. To formally state this result, let $t$ be an indeterminate, let $p$ be an odd prime, and let $K/\mathbb{F}_p(t)$ be a finite separable extension. Extending the usual derivative $\frac{d}{dt}$ on $\mathbb{F}_p(t)$ to $K$ via implicit differentiation, we let $\beta'$ denote the derivative of $\beta\in K$. Given a polynomial $\phi(x)\in K[x]$ of degree $d\geq3$, write 
\[\phi(x)=A_0x^d+A_{1}x^{d-1}+\dots A_{d-1}x+A_d,\;\;\;\;\;\;\;\; A_i\in K \vspace{.05cm}.\]
Then we have the following important quantities (c.f. \cite[Theorem 1.1]{Riccati}) associated to $\phi$: 
\begin{enumerate}[topsep=8pt, partopsep=8pt, itemsep=8pt]   
\item[\textup{(1)}] $\delta_\phi:=2dA_0A_{2}-(d-1)A_{1}^2$,    
\item[\textup{(2)}] $b_\phi:=(dA_0^2A_2' - (d-1)A_0A_1A_1' - dA_0A_2A_0' + (d-1)A_1^2A_0')/\delta_\phi$
\item[\textup{(3)}] $f_\phi=(d^2A_0^2A_2' - d(d-1)A_0A_1A_1' - d(d-2)A_0A_2A_0' + (d(d-2)+1)A_1^2A_0')/\delta_\phi$
\end{enumerate}
Generalizing \cite[Theorem 1.1]{Riccati} to dynamical systems generated by a finite set of maps, we prove the following: 
\begin{theorem}\label{thm:primdivff} Let $K/\mathbb{F}_p(t)$ be a function field and suppose that $S=\{\phi_1,\phi_2,\dots, \phi_s\}$ is a set of polynomials of degree at least $3$, defined over $K$, and satisfying:
\begin{enumerate}[topsep=8pt, partopsep=8pt, itemsep=10.5pt] 
\item[\textup{(1)}] $\deg(\phi)\deg(\psi)\delta_\phi\delta_\psi\big(b_\phi-f_\psi\big)\neq0$ for all $\phi,\psi\in S$, 
\item[\textup{(2)}] $0\not\in\PrePer(\Phi_S)$.  
\end{enumerate} 
Then $\mathcal{Z}(\gamma,P)$ is finite for all $(\gamma,P)\in G_S$.   
\end{theorem}   
\begin{remark} Theorem \ref{thm:primdivff} is quite broadly applicable from the point of view of algebraic geometry: let $d=\max\{\deg(\phi): \phi\in S\}$ and view $S$ as a point of $(\mathbb{A}^{d+1})^s$ by associating to each polynomial in $S$ its $(d+1)$-tuple of coefficients. Then the sets $S$ satisfying condition (1) of Theorem \ref{thm:primdivff} are Zariski dense in $(\mathbb{A}^{d+1})^s$; compare to \cite[Remark 1.1]{Riccati}. Moreover, $\PrePer(\Phi_S)\cap\mathbb{P}^1(K)$ is finite; see Lemma \ref{lem:particularcanonicalheight} or \cite[Corollary B]{Kawaguchi}.  
\end{remark}  \vspace{.1cm}  
\textbf{Acknowledgments:} We thank Patrick Ingram and Joseph Silverman for helpful conversations and for making us aware of previous work on random canonical heights in \cite{Kawaguchi}.      
\section{Expected Canonical Heights}\label{sec:hts}  
To prove Theorem \ref{thm:htsiid}, we need a generalization of Tate's telescoping argument for the usual canonical height. To do this, we recall that the functoriality of heights implies that 
\[h_{V,\eta}\circ\phi=\alpha_\phi h_{V,\eta} +O_{V,\phi}(1)\;\;\;\;\;\text{for $\phi\in S$.}\]    
In particular, we have the following bound for height controlled families:  
\begin{lemma}{\label{lem:uniformbd}} Let $S$ be height controlled with respect to $\eta$, let $C:=\sup\big\{O_{V,\phi}(1)\big\}$, and let $\alpha:=\inf\{\alpha_\phi\}$. If $\rho_r\in\Phi_r$ for $r\geq1$, then \vspace{.05cm}  
\[\bigg|\frac{h_{V,\eta}(\rho_r(Q))}{\deg_\eta(\rho_r)}-h_{V,\gamma}(Q)\bigg|\leq \frac{C}{\alpha-1} \vspace{.05cm} \]
for all $Q\in V(\overline{K})$. In particular, this bound is independent of $\rho_r$, $r$, and $Q$.  
\end{lemma} 
\begin{proof} Suppose that $\rho_r=\theta_r\circ\theta_{r-1}\dots \circ\theta_1$ for $\theta_i\in S$, and let $\theta_0$ to be the identity map on $V$. Then define 
\[\rho_{i}:=\theta_i\circ\theta_{i-1}\dots \circ\theta_1\circ\theta_0 \;\;\;\; \text{for $0\leq i\leq r$}.\vspace{.05cm}\]
Note, that $\rho_0=\theta_0$ is the identity map. In particular, inspired by Tate's telescoping argument, we rewrite \vspace{.05cm} 
\begin{equation}{\label{Tate}}
\begin{split}
\bigg|\frac{h_{V,\eta}(\rho_r(Q))}{\deg_\eta(\rho_r)}-h_{V,\gamma}(Q)\bigg|&=\bigg|\sum_{i=0}^{r-1}\frac{h_{V,\eta}(\rho_{r-i}(Q))}{\deg_\eta(\rho_{r-i})}- \frac{h_{V,\eta}(\rho_{r-i-1}(Q))}{\deg_\eta(\rho_{r-i-1})}\bigg|\\[5pt]
&\leq \sum_{i=0}^{r-1}\bigg|\frac{h_{V,\eta}(\rho_{r-i}(Q))}{\deg_\eta(\rho_{r-i})}- \frac{h_{V,\eta}(\rho_{r-i-1}(Q))}{\deg_\eta(\rho_{r-i-1})}\bigg| \\[5pt]
&=\sum_{i=0}^{r-1}\frac{\Big|h_{V,\eta}(\rho_{r-i}(Q))-\deg_\eta(\theta_{r-i})h_{V,\eta}(\rho_{r-i-1}(Q))\Big|}{\deg_\eta(\rho_{r-i})} \\[5pt] 
&\leq \sum_{i=0}^{r-1}\frac{C}{\deg_\eta(\rho_{r-i})} \\[5pt]
&\leq \sum_{i=1}^{r}\frac{C}{\alpha^i} \\[5pt]
&\leq \sum_{i=1}^{\infty}\frac{C}{\alpha^i}=\frac{C}{\alpha-1}.  
\end{split} 
\end{equation}  
This completes the proof of Lemma \ref{lem:uniformbd}.
\end{proof} 
It now follows easily that the canonical height $\hat{h}_{V,\eta,P}(\gamma)$ is well-defined and is uniformly bounded (as we vary possible paths $\gamma\in\Phi_S$) by the the Weil height: 
\begin{lemma}\label{lem:particularcanonicalheight} Let $P\in V(\overline{K})$ and let $\gamma\in\Phi$. Then the canonical height, 
\[\hat{h}_{V,\eta,P}(\gamma):=\lim_{n\rightarrow\infty}\frac{h_{V,\eta}(\gamma_n\cdot P)}{\deg_\eta(\gamma_n)},\] 
is well defined. Moreover, $|\hat{h}_{V,\eta,P}(\gamma)-h_{V,\eta}(P)|\leq\frac{C}{\alpha-1}$ for all $P\in V(\overline{K})$ and $\gamma\in\Phi$.     
\end{lemma}
\begin{proof} This is a simple application of Lemma \ref{lem:uniformbd}. Let $\gamma=(\theta_n)_{n\geq1}^\infty$, and write \[\gamma_r:=\theta_r\circ\theta_{r-1}\dots\circ\theta_1\;\;\;\;\;\text{for $r>0$}.\] 
Likewise, for $n>m>0$, let $\rho=(\gamma_n\mysetminus\gamma_m):=\theta_{n}\circ\theta_{n-1}\dots \theta_{m+1}$. In particular, we see that \vspace{.1cm}  
\begin{equation}{\label{particularcanonicalheight}} 
\begin{split} 
\bigg|\frac{h_{V,\eta}(\gamma_n\cdot P)}{\deg_\eta(\gamma_n)}-\frac{h_{V,\eta}(\gamma_m\cdot P)}{\deg_\eta(\gamma_m)}\bigg|&=\frac{1}{\deg_\eta(\gamma_m)}\bigg|\frac{h_{V,\eta}(\rho\cdot(\gamma_m\cdot P))}{\deg_\eta(\rho)}-h_{V,\eta}(\gamma_m\cdot P))\bigg|\\[5pt]
&\leq\frac{C}{\alpha^{m}(\alpha-1)}.
\end{split} 
\end{equation}
Here we apply Lemma \ref{lem:uniformbd} to the map $\rho$ and the basepoint $Q:=\gamma_m\cdot P$, and we use that $\deg(\gamma_m)\geq \alpha^m$. Letting $m$ grow arbitrarily large, we see that the distance between the $n$th and $m$th term of the sequence defining $\hat{h}_{V,\eta,P}(\gamma)$ goes to zero. In particular, this sequence is Cauchy and therefore converges. As for the bound $|\hat{h}_{V,\eta,P}(\gamma)-h_{V,\eta}(P)|\leq\frac{C}{\alpha-1}$, let $m=0$ and $n\rightarrow\infty$ in (\ref{particularcanonicalheight}).      
\end{proof} 
\begin{remark}\label{rem:ht} As with the usual canonical height associated to ample divisors $\eta$, we have that $\hat{h}_{V,\eta,P}(\gamma)=0$ if and only if $\Orb_\gamma(P)$ is finite. This follows  readily from Lemma \ref{lem:particularcanonicalheight} and the simple identity 
\[\hat{h}_{V,\eta,\gamma_m\cdot P}(\gamma\mysetminus \gamma_m)=\deg(\gamma_m)\,\hat{h}_{V,\eta,P}(\gamma)\;\;\;\text{for all $\gamma\in\Phi_S$ and $m\geq1$};\]   
here, $\gamma\mysetminus\gamma_m:=(\theta_n)_{n\geq m+1}$ is the $m$th shift of $\gamma=(\theta_n)_{n\geq1}\in\Phi_S$. In particular, if $\hat{h}_{V,\eta,P}(\gamma)=0$ then $\hat{h}_{V,\eta}(\gamma_m\cdot P)$ is absolutely bounded by Lemma \ref{lem:particularcanonicalheight}. In particular, the dynamical orbit $\Orb_\gamma(P)=\{\gamma_m\cdot P: m\geq1\}$ is finite by Northcott's Theorem when $\eta$ is ample. On the other hand, if $\Orb_\gamma(P)$ is finite, then $\hat{h}_{V,\eta}(\gamma_m\cdot P)$ is bounded and $\hat{h}_{V,\eta,P}(\gamma)=0$ as claimed.         
\end{remark}
We now study the expected value of these canonical height functions on the probability space $(\Phi_S,\mathcal{F},\nu)$ of i.i.d sequences of elements of $S$ distributed according to $\nu_1$; see \cite[Theorem 10.4]{ProbabilityText}. In what follows, we suppress $S$ and $\mathcal{F}$ in the notation and simply write $\Phi$ and $\nu$ where appropriate.  
\begin{proof}[(Proof of Theorem \ref{thm:htsiid})] Fix $P\in V(\overline{K})$ and consider the sequence of random variables $h_{V,\eta,n,P}:\Phi\rightarrow \mathbb{R}$ defined by 
\[h_{V,\eta,P,n}(\gamma):= \frac{h_{V,\eta}(\gamma_n\cdot P)}{\deg_\eta(\gamma_n)}. \]
It follows from the definition of $(\Phi,\nu)$ that each $h_{V,\eta,P,n}$ is $\nu$-measurable; see, for instance,  \cite[Theorem 10.4]{ProbabilityText}. On the other hand, the random variable $\hat{h}_{V,\eta,P}:\Phi\rightarrow\mathbb{R}$ is the pointwise limit of the $\{h_{V,\eta,P,n}\}_{n\geq1}$:  
\[\hat{h}_{V,\eta,P}(\gamma)=\lim_{n\rightarrow\infty}h_{V,\eta,P,n}(\gamma)\;\;\;\;\;\; \text{for all $\gamma\in\Phi$}.\]
Moreover, the functions $\{h_{V,\eta,P,n}\}_{n\geq1}$ are absolutely bounded; see Lemma \ref{lem:uniformbd}. Hence, the Lebegue dominated convergence theorem \cite[Theorem 9.1]{ProbabilityText} implies that $\hat{h}_{V,\eta,P}$ is integrable and that 
\begin{equation}{\label{dominatedconvergence}} 
\mathbb{E}_\nu\big[\hat{h}_{V,\eta}\big](P):=\int_{\Phi}\hat{h}_{V,\eta,P}(\gamma)\,d\nu=\lim_{n\rightarrow\infty}\int_{\Phi} h_{V,\eta,P,n}(\gamma)\,d\nu.
\end{equation} 
Hence, $\mathbb{E}_\nu\big[\hat{h}_{V,\eta}\big](P)$ is well defined for every $P\in V(\bar{K})$. As for properties $(a)$ and $(b)$, note that Lemma \ref{lem:particularcanonicalheight} and the triangle inequality (for integrals) imply that 
\begin{equation*} 
\begin{split} 
\Big|\mathbb{E}_\nu\big[\hat{h}_{V,\eta}\big](P)-h_{V,\eta}(P)\Big|&=\bigg|\int_{\Phi}\hat{h}_{V,\eta,P}(\gamma)\,d\nu-h_{V,\eta}(P)\cdot\int_{\Phi} 1\,d\nu\,\bigg|=\bigg|\int_{\Phi}\hat{h}_{V,\eta,P}(\gamma)-h_{V,\eta}(P)\,d\nu\bigg|  \\[2pt]  
&\leq\int_{\Phi}\Big|\hat{h}_{V,\eta,P}(\gamma)-h_{V,\eta}(P)\Big|\,d\nu\leq \frac{C}{\alpha-1}.  
\end{split} 
\end{equation*} 
Hence, $\mathbb{E}_\nu\big[\hat{h}_{V,\eta}\big]=h_{V,\eta}+O(1)$ as claimed. As for the transformation property in (b), we have first that
\[\mathbb{E}_{\nu_n}\bigg[\frac{\mathbb{E}_\nu\,[\hat{h}_{V,\eta}](\gamma_n\cdot P)}{\deg_\eta(\gamma_n)}\bigg]:=\int_{\Phi_n}\frac{\mathbb{E}[\hat{h}_{V,\eta}](\gamma_n\cdot P)}{\deg_\eta(\gamma_n)}\,d\nu_n:=\sum_{\gamma_n\in\Phi_n}\frac{\mathbb{E}[\hat{h}_{V,\eta}](\gamma_n\cdot P)}{\deg_\eta(\gamma_n)}\nu_n(\gamma_n)\] 
by definition. In particular, the Lebegue dominated convergence theorem \cite[Theorem 9.1]{ProbabilityText}, Fubini's Theorem \cite[Theorem 10.3]{ProbabilityText}, and \cite[Corollary 10.2]{ProbabilityText} together imply that 
\begin{align*}
\mathbb{E}_{\nu_n}\bigg[\frac{\mathbb{E}\,[\hat{h}_{V,\eta}](\gamma_n\cdot P)}{\deg_\eta(\gamma_n)}\bigg]&= \int_{\Phi_n}\frac{\displaystyle\lim_{m\rightarrow\infty}\int_{\Phi_m}\frac{h_{V,\eta}(\gamma_m\cdot (\gamma_n\cdot P))}{\deg_\eta(\gamma_m)} \,d\nu_m}{\deg_\eta(\gamma_n)} \,d\nu_n && \text{(\ref{dominatedconvergence}) and \cite[Corollary 10.2]{ProbabilityText}}\\[3pt]
&=\lim_{m\rightarrow\infty} \int_{\Phi_n}\int_{\Phi_m} \frac{h_{V,\eta}(\gamma_m\cdot (\gamma_n\cdot P))}{\deg_\eta(\gamma_m)\deg_\eta(\gamma_n)} \,d\nu_m\,d\nu_n && \text{\cite[Theorem 9.1]{ProbabilityText}}\\[3pt]
&=\lim_{m\rightarrow\infty} \int_{\Phi_n}\int_{\Phi_m} \frac{h_{V,\eta}((\gamma_m\circ\gamma_n)\cdot P))}{\deg_\eta(\gamma_m\circ\gamma_n)} \,d\nu_m\,d\nu_n\\[3pt]
&=\lim_{m\rightarrow\infty}\int_{\Phi_{n+m}} \frac{h_{V,\eta}(\gamma_{n+m}\cdot P)}{\deg_\eta(\gamma_{n+m})}\,d\nu_{n+m} &&\text{\cite[Theorem 10.3]{ProbabilityText}}\\[5pt]
&=\lim_{m\rightarrow\infty} \int_{\Phi} h_{V,\eta,P,n+m}(\gamma)\,d\nu &&\text{\cite[Corollary 10.2]{ProbabilityText}}  \\[3pt]
&=\int_{\Phi}\hat{h}_{V,\eta,P}(\gamma)\,d\nu &&\text{(\ref{dominatedconvergence})} \\[3pt]
&=\mathbb{E}\big[\hat{h}_{V,\eta}\big](P). \vspace{.2cm}  
\end{align*} 
On the other hand, $\nu_n(\gamma_n)/\deg_\eta(\gamma_n)=(d_{\nu_1,\eta})^{-n}\,\nu_n^*(\gamma_n)$ for all $\gamma_n\in\Phi_n$ by definition of $\nu_n^*$. Therefore, we deduce that
\begin{equation} 
\begin{split}
\mathbb{E}\big[\hat{h}_{V,\eta}\big](P)&=\mathbb{E}_{\nu_n}\bigg[\frac{\mathbb{E}\,[\hat{h}_{V,\eta}](\gamma_n\cdot P)}{\deg_\eta(\gamma_n)}\bigg]=\sum_{\gamma_n\in\Phi_n}\frac{\mathbb{E}[\hat{h}_{V,\eta}](\gamma_n\cdot P)}{\deg_\eta(\gamma_n)}\nu_n(\gamma_n)\\[3pt] 
&=(d_{\nu_1,\eta})^{-n}\hspace{-.2cm}\sum_{\gamma_n\in\Phi_n}\mathbb{E}[\hat{h}_{V,\eta}](\gamma_n\cdot P)\nu_n^*(\gamma_n)=(d_{\nu_1,\eta})^{-n}\, \mathbb{E}_{\nu_n^*}\Big[\mathbb{E}[\hat{h}_{V,\eta}](\gamma_n\cdot P)\Big] \\[1pt] 
\end{split}  
\end{equation} 
as claimed. This is the analog of the usual transformation property for canonical heights defined by a fixed endomorphism.  
\end{proof} 
\begin{remark} We note that if $\deg(\phi)=d$ for all $\phi\in S$, then $d_{\nu,\eta}=d$, $\nu^*_k=\nu_k$, and we obtain the transformation formula $\mathbb{E}_{\nu_k}\Big[\mathbb{E}_{\nu}\big[\hat{h}_{V,\eta}\big](\gamma_k\cdot P)\Big]=d^{k}\;\mathbb{E}_{\nu}\big[\hat{h}_{V,\eta}\big](P).$ 
\end{remark} 
\begin{remark} We note that for each $P\in V$ the variance of the distribution of $\hat{h}_{V,\eta,P}(\gamma)$ as we vary over sequences $\gamma\in\Phi_S$ is absolutely bounded by Lemma \ref{lem:particularcanonicalheight} and Popoviciu's inequality:  
\[\sigma_{V,\eta,\nu,P}^2:=\int_{\Phi_S}\Big(\hat{h}_{V,\eta,P}-\mathbb{E}_\nu[\hat{h}_{V,\eta}](P)\Big)^2\,d\nu\leq \frac{1}{4}\Big(\frac{2C}{\alpha-1}\Big)^2.\]
\end{remark} 
\begin{example}[Finite collections] Let $V:=\mathbb{P}^N$, let $S=\{\phi_1,\phi_2,\dots,\phi_s\}$ be any finite collection of endomorphisms of degree at least two, and let $\nu_1$ be any probability measure on $S$. Then any divisor class $\eta\in \Pic(\mathbb{P}^N)\cong\mathbb{Z}$ satisfies $\phi_i^*(\eta)=\deg(\phi_i)\eta$ for all $1\leq i\leq s$. Moreover, there are constants $O_{\eta,i}(1)$ for each $1\leq i\leq s$ such that  
\[h_{\mathbb{P}^N,\eta}\circ\phi_i=\deg(\phi_i) h_{V,\eta} +O_{\eta,i}(1).\]  
Therefore, $S$ is height controlled with respect to any divisor: $C_\eta:=\max\{O_{\eta,i}(1)\}$ is finite. In particular, one could take: $V:=\mathbb{P}^1$, $S:=\{\phi_1, \phi_2\}$ to be any two rational maps in one-variable (of degree at least $2$), and $\nu_1(\phi_1)=1/2=\nu_1(\phi_2)$: that is, we flip a fair coin to determine which map to apply at every stage of composition.    
\end{example}  
\begin{example}[Unicritical maps of bounded height]{\label{unicrit}} Let $V:=\mathbb{P}^1$, let $\eta=\infty$, let $B>0$, and consider the set of functions 
\[S_B:=\big\{\phi_{d,c}(z)=z^d+c\,:\, c\in\mathbb{Z},\; |c|\leq B,\; d>1\big\}.\]
Then, \cite[Lemma 12]{Ingram} implies that 
\begin{equation}\label{htbdunicritical}
\;\;\;|h_\eta(\phi_{d,c}(P))-dh_\eta(P)|\leq \log|2c|\leq \log|2B|\;\;\;\;\;\;\; \text{for all $P\in\mathbb{P}^1(\mathbb{Q})\mysetminus\{\eta\}=\mathbb{Q}$.}
\end{equation} 
In particular, (for simplicity) we consider only rational basepoints. Hence, (\ref{htbdunicritical}) implies that $S_B$ is height controlled with respect to $\eta$. From here, we can take any probability measure $\nu_1$ we like on $S_B$ (note that $S_B$ is set theoretically just $\mathbb{N}^{2B+1}$). For instance, one could consider any of the following well-known probability measures: \\[4pt]
\underline{\emph{Geometric:}} Let $r\in(0,1)$ and let  $\nu_{B,r,1}$ be the measure on $(S_B,2^{S_B})$ generated by 
\[\nu_{B,r,1}(\phi_{d,c})=\frac{(1-r)r^{d-2}}{2B+1}.\] 
Then one checks via the summation formula for geometric series that $\sum \nu_{B,r,1}(\phi_{d,c})=1$. In particular, the data $\big(\mathbb{P}^1(\mathbb{Q}), \eta, S_B, \nu_{B,r,1}\big)$ satisfies the hypothesis of Theorem \ref{thm:htsiid}. \\[5pt]
\underline{\emph{Poisson:}} Let $\lambda>0$ and let $\nu_{B,\lambda,1}$ be the measure on $(S_B,2^{S_B})$ generated by 
\[\nu_{B,\lambda,1}(\phi_{d,c})=\frac{e^{-\lambda}\lambda^{d-2} }{(2B+1)(d-2)!}.\] 
Then one checks via the exponential summation formula that $\sum \nu_{B,\lambda,1}(\phi_{d,c})=1$. In particular, the data $\big(\mathbb{P}^1(\mathbb{Q}), \eta, S_B, \nu_{B,\lambda,1}\big)$ satisfies the hypothesis of Theorem \ref{thm:htsiid}.

\begin{remark}\label{htcont} Likewise, given any probability measure on the set of integers $|c|\leq B$, one can twist the usual Geometric and Poisson distributions to form new probability measures on $S_B$ and study the corresponding expected canonical heights. Even more generally, for any finite set of rational maps $S$, the set $\bar{S}=\{\phi\circ x^d\,: \phi\in S, d\geq2\}$ is an infinite, height controlled family (generalizing the unicritical maps of bounded height).  
\end{remark}   
\end{example}
We now prove a dynamical application of the expected canonical height function (analogous to the characterization of preperiodic points as the kernel of the usual canonical height).    
\begin{proof}[(Proof of Corollary \ref{cor:hteq0})] If there exists a finite, $S$-stable subset $F_P$ containing $P$, then $\Orb_\gamma(P)$ is contained in $F_P$ for all $\gamma\in\Phi_S$; hence, 
\[\nu\big(\{\gamma\in\Phi_S:\, \Orb_\gamma(P)\; \text{is finite}\;\}\big)=1.\] 
Therefore, (1) $\implies$ (2). On the other hand, if (2) holds, then we see immediately that  
\[\nu\big(\{\gamma\in\Phi_S:\, \hat{h}_{V,\eta,P}(\gamma)=0\}\big)=1.\] 
In particular, the expected canonical height must vanish: $\mathbb{E}_\nu[\hat{h}_{{V,\eta}}](P)=\int_{\Phi_S}\hat{h}_{V,\eta,P}(\gamma)\,d\nu=0$, and (2) $\implies$ (3). Finally, suppose that $\mathbb{E}_\nu[\hat{h}_{{V,\eta}}](P)=0$. Then the transformation law in property (b) of Theorem \ref{thm:htsiid} implies that 
\[\mathbb{E}_{\nu_k^*}\Big[\mathbb{E}_{\nu}\big[\hat{h}_{V,\eta}\big](\gamma_k\cdot P)\Big]=(d_{\nu,\eta})^{k}\,\mathbb{E}_{\nu}\big[\hat{h}_{V,\eta}\big](P)=0\]
for all $k\geq1$. However, $\Phi_{S,k}$ is countable, so that 
\begin{equation}\label{eq0} 
0=\mathbb{E}_{\nu_k^*}\Big[\mathbb{E}_{\nu}\big[\hat{h}_{V,\eta}\big](\gamma_k\cdot P)\Big]=\sum_{\gamma_k\in\Phi_k}\mathbb{E}[\hat{h}_{V,\eta}](\gamma_k\cdot P)\nu_k^*(\gamma_k). 
\end{equation} 
Moreover, since $\eta$ is ample, \cite[Theorem 2.3 (3)]{Kawaguchi} implies that $\hat{h}_{V,\eta,Q}$ is non-negative for all $Q\in\mathbb{P}^1(\overline{K})$. Therefore, $\mathbb{E}_\nu[\hat{h}_{V,\eta}](\gamma_k\cdot P)$ is non-negative for all $\gamma_k\in\Phi_{S,k}$ and all $k\geq1$. On the other hand, $\nu_k^*(\gamma_k)$ is positive, since $\deg_\eta(\gamma_k)$ is positive and $\nu_1$ is a strictly positive measure. Hence, (\ref{eq0}) implies that $\mathbb{E}_\nu[\hat{h}_{V,\eta}](\gamma_k\cdot P)=0$ for all $\gamma_k\in\Phi_{S,k}$ and all $k\geq1$. However, $\mathbb{E}_\nu[\hat{h}_{V,\eta}]=h_{V,\eta}+O(1)$, so that 
\begin{equation}\label{total orbit2} 
F_P:=\bigcup_{\gamma\in\Phi_{S}}\bigcup_{k\geq0}\,(\gamma_k\cdot P)\subseteq V(K(P))
\end{equation}  
must be a set of bounded height (with respect to $h_{V\eta}$); here $K(P)/K$ is the field of definition of $P$. In particular, Northcott's theorem and the fact that $\eta$ is ample imply that $F_P$ is finite. Moreover, $F_P$ is $S$-stable. Therefore (3) $\implies$ (1), completing the argument.     
\end{proof} 
\begin{remark} To summarize, if $\eta$ is an ample divisor, then we have defined a height function $\mathbb{E}_\nu[\hat{h}_{V,\eta}]$ on $V$ that characterizes the existence of a finite subset that is simultaneously stable for all maps in $S$. In fact, we could formally define a map $\mathbb{E}_\nu[\hat{h}_{V,\eta}]: \text{Fin}(V)\rightarrow \mathbb{R}$ on the set $\text{Fin}(V)$ of finite subsets of $V$ given by $\mathbb{E}_\nu[\hat{h}_{V,\eta}](F)=\sum_{P\in F}\mathbb{E}_\nu[\hat{h}_{V,\eta}](P)$, and note that $\mathbb{E}_\nu[\hat{h}_{V,\eta}](F)=0$ if and only if $F$ is $S$-stable. 
\end{remark} 
\begin{remark} Note that Corollary \ref{cor:hteq0} does not depend on the probability measure $\nu_1$ on $S$.   
\end{remark} 
\begin{proof}[(Proof of Corollary \ref{cor:julia})] If $P\in\mathbb{P}^1(\overline{\mathbb{Q}})$ is such that $\mathbb{E}_\nu[\hat{h}](P)=0$, then Corollary \ref{cor:hteq0} implies that $P\in\PrePer(\phi)$ for all $\phi\in S$; however, note that the converse is false in general: $P=0$ and $S=\{x^2-2,x^2+x\}$; In any case, if $\big\{P\in\mathbb{P}^1(\overline{\mathbb{Q}})\,:\,\mathbb{E}_\nu[\hat{h}](P)=0\big\}$ is infinite, then $\bigcap_{\phi\in S}\PrePer(\phi)$ is infinite, and (1) implies (2). On the other hand, if $\bigcap_{\phi\in S}\PrePer(\phi)$ is infinite, then $\PrePer(\phi)\cap\PrePer(\psi)$ is infinite for all $\phi,\psi\in S$. Hence, we deduce that $\PrePer(\phi)=\PrePer(\psi)$ for all $\phi,\psi\in S$ by \cite[Theorem 1.2]{Baker-DeMarco}. Therefore, (2) implies (3). Finally, suppose that $\PrePer(\phi)=\PrePer(\psi)$ for all $\psi\in S$, and take any point $P\in\PrePer(\phi)$. Then
\[F_P:=\bigcup_{\gamma\in\Phi_{S}}\bigcup_{k\geq0}\,(\gamma_k\cdot P)\subseteq \PrePer(\phi)\cap \mathbb{P}^1(K(P));\]
here we use that $\PrePer(\psi)$ is $\psi$-stable and that $\PrePer(\phi)=\PrePer(\psi)$ for all $\psi\in S$. However, $\PrePer(\phi)$ is a set of bounded height by Northcott's theorem. Therefore, $F_P$ is a finite set since $[K(P):K]$ is a finite extension. On the other hand, $F_P$ is clearly $S$-stable, so that Corollary \ref{cor:hteq0} implies that $\mathbb{E}_\nu[\hat{h}](P)=0$. Hence, we have shown that \[\PrePer(\phi)\subseteq\big\{P\in\mathbb{P}^1(\overline{\mathbb{Q}})\,:\,\mathbb{E}_\nu[\hat{h}](P)=0\big\}.\] 
In particular, since $\PrePer(\phi)$ is an infinite set by \cite[Exercise 1.18]{SilvDyn}, we se that (3) implies (1) as claimed.            
\end{proof} 
\section{Primitive prime divisors in generalized orbits}\label{sec:primdiv}
To begin this section, let $K$ be a number field and let $\mathfrak{o}_K$ be the ring of integers of $K$. Given a prime ideal $\mathfrak{p}\subset \mathfrak{o}_K$, we let $k_\mathfrak{p}:=\mathfrak{o}_K/\mathfrak{p}\mathfrak{o}_K$ and $N_\mathfrak{p}:=(\log\#k_\mathfrak{p})/[K:\mathbb{Q}]$ be the residue field and local (logarithmic) degree of $\mathfrak{p}$ respectively, and define the Weil height \cite[\S3.1]{SilvDyn} of $\alpha\in K^*$ by  
\[h(\alpha)=-\sum_{\mathfrak{p}}\min(v_\mathfrak{p}(\alpha),0) N_\mathfrak{p}+\frac{1}{[K:\mathbb{Q}]}\sum_{\sigma:K\rightarrow\mathbb{C}}\max(\log|\sigma(\alpha)|,0). \]
In particular, it follows from the product formula \cite[Proposition 3.3]{SilvDyn} that \vspace{.1cm} 
\begin{equation}{\label{ht:lbd}}
\sum_{v_\mathfrak{p}(\alpha)>0}\hspace{-.2cm}v_\mathfrak{p}(\alpha)N_\mathfrak{p} \leq h(\alpha)\;\;\;\;\;\;\text{and}\;\;\;\;\;\; -\sum_{v_\mathfrak{p}(\alpha)<0}\hspace{-.2cm}v_\mathfrak{p}(\alpha)N_\mathfrak{p}\leq h(\alpha)
\end{equation} 
for all $\alpha\in K^*$. As in \cite{Tucker} and \cite{me:abc}, the main tool we use to study the prime factors of elements of orbits is the (Roth) $abc$-conjecture.    
\begin{conjecture}[$abc$-conjecture]\label{conj:abc} Let $K$ be a number field. Then for any $\epsilon>0$, there exists a constant $C_{K,\epsilon}$ such that for all $a,b,c\in K^*$ satisfying $a+b=c$, we have that \vspace{.05cm} 
\[h(a,b,c)<(1+\epsilon)(\text{rad}(a,b,c))+C_{K,\epsilon}. \vspace{.05cm} \]
Here \text{rad}(a,b,c) is a suitably defined radical of the tuple $(a,b,c)$; see \cite[\S 3]{Tucker}.  
\end{conjecture}
In fact, the key idea we use here is that the $abc$-conjecture implies that polynomial values are reasonably square-free in the following sense; see \cite[Proposition 3.4]{Tucker}
\begin{proposition}\label{prop:abc} Let $F(x)\in\mathfrak{o}_K[x]$ be a polynomial of degree at least $3$ without repeated factors. Then for all $\epsilon>0$ there exists a constant $C_{f,\epsilon}$ such that: \vspace{.05cm} 
\[\sum_{v_\mathfrak{p}(F(z))>0}\hspace{-.4cm}\text{N}_\mathfrak{p}\geq (\deg(F)-2-\epsilon)h(z)+C_{f,\epsilon}\;\;\;\text{for all $z\in K$.} \]
\end{proposition} 
As a reminder, in our study of prime factors in dynamical orbits, we restrict our attention to pairs $(\gamma,P)\in\Phi_S\times\mathbb{P}^1(K)$ in the following set:  
\[G_S:=\Big\{(\gamma,P)\in\Phi_S\times\mathbb{P}^1(K):\,\hat{h}_\gamma(P)>0\;\text{and}\; 0,\infty\not\in \Orb_{\gamma}(P)\cup\Orb_\gamma(0)\Big\}.\]
\begin{proof}[(Proof of Theorem \ref{thm:primdiv})] For each $\phi\in S$, fix a representation $\phi(x)=\frac{F_\phi(x)}{G_\phi(x)}$ for some coprime polynomials $F_\phi,G_\phi\in\mathfrak{o}_K[x]$, and let \vspace{.1cm} 
\[\mathfrak{s}:=\big\{\frak{p}\,:\, v_\mathfrak{p}(\Res(F_\phi,G_\phi))\neq 0, \; \phi\in S\big\}; \vspace{.1cm}\] 
here $\Res(F_\phi,G_\phi)$ denotes the resultant of $F_\phi$ and $G_\phi$. In particular, $\mathfrak{s}$ is finite and $\gamma_m$ has good reduction at all $\mathfrak{p}\not\in\mathfrak{s}$ for all $m\geq1$ and all $\gamma\in\Phi_S$; see \cite[Theorem 2.15]{SilvDyn}. To ease notation, for a given sequences $\gamma=(\theta_n)_{n\geq1}$, write $F_n=F_{\theta_n}$ and $G_n=G_{\theta_n}$ for any $n\geq1$. Note that Proposition \ref{prop:abc} implies that for all $\epsilon>0$ there exists $C_{S,\epsilon}=\min\{C_{F_n,\epsilon}\}$ such that: 
\begin{equation}\label{abc1}
(\deg(F_n)-2-\epsilon)h(\gamma_{n-1}\cdot P)+C_{S,\epsilon}\;\leq\hspace{-.4cm}\sum_{\hspace{.2cm}v_\mathfrak{p}(F_n(\gamma_{n-1}\cdot P))>0}\hspace{-.8cm}N_\mathfrak{p}\;.
\end{equation} 
Here we use assumptions (1) and (2) of Theorem \ref{thm:primdiv} to deduce that $F_n$ is square-free and has degree at least $3$; see also Remark \ref{sqf}. Moreover, (\ref{abc1}) holds for all $P$, $\gamma$, and $n$. 

Now assume that $n\in\mathcal{Z}(\gamma,P)$. To study the prime ideals defining the right hand side of (\ref{abc1}), assume that $v_\mathfrak{p}(F_n(\gamma_{n-1}\cdot P))>0$. If in addition $\mathfrak{p}\not\in\mathfrak{s}$ and $v_\mathfrak{p}(\gamma_{n-1}\cdot P)\geq0$, then \vspace{.1cm}
\begin{equation}\label{abc2}
v_\mathfrak{p}(\gamma_{n}\cdot P)=v_\mathfrak{p}(F_n(\gamma_{n-1}\cdot P))-v_\mathfrak{p}(G_n(\gamma_{n-1}\cdot P))=v_\mathfrak{p}(F_n(\gamma_{n-1}\cdot P))>0 \vspace{.1cm}
\end{equation} 
by standard properties of the resultant; see \cite[Theorem 2.15]{SilvDyn}. On the other hand, we have assumed that $n\in\mathcal{Z}(\gamma,P)$, so that (\ref{abc2}) implies that $v_\mathfrak{p}(\gamma_{m}\cdot P)>0$ for some $1\leq m\leq n-1$. Moreover, \cite[Exercise 2.12]{SilvDyn} implies that all possible compositional combinations of elements of $S$ have good reduction at $\mathfrak{p}$, so that that \vspace{.1cm}
\begin{equation}\label{abc3} (\gamma_n\mysetminus\gamma_m)\cdot0\equiv(\gamma_n\mysetminus\gamma_m)\cdot(\gamma_m\cdot P)\equiv\gamma_n\cdot P\equiv 0\Mod{\mathfrak{p}}; \vspace{.1cm}
\end{equation}
see \cite[Theorem 2.18]{SilvDyn}. Therefore, (\ref{ht:lbd}), (\ref{abc1}) and (\ref{abc3}) together imply that
\begin{equation}{\label{abc4}}
\Scale[0.88]{
\begin{split} 
(\deg(F_n)-2-\epsilon)h(\gamma_{n-1}\cdot P)+C_{S,\epsilon}&\leq\underset{v_\mathfrak{p}(\gamma_m\cdot P)>0}{\sum_{m=1}^{\lfloor\frac{n}{2}\rfloor}}\hspace{-.4cm}N_\mathfrak{p}\;\;+\hspace{-.15cm}\underset{v_\mathfrak{p}((\gamma_n\mysetminus\gamma_m)\cdot 0)>0}{\sum_{m=\lfloor\frac{n}{2}\rfloor}^{n-1}}\hspace{-.7cm}N_\mathfrak{p}\;+\;\sum_{\mathfrak{p}\in \mathfrak{s}}N_\mathfrak{p}\;+\hspace{-.25cm}\sum_{\;\;v_\mathfrak{p}(\gamma_{n-1}\cdot P)<0}\hspace{-.6cm}N_\mathfrak{p} \\[6pt]
&\leq \sum_{m=1}^{\lfloor\frac{n}{2}\rfloor}h(\gamma_m\cdot P)+\sum_{m=\lfloor\frac{n}{2}\rfloor}^{n-1}\hspace{-.2cm}h((\gamma_n\mysetminus\gamma_m)\cdot 0) +h(\gamma_{n-1}\cdot P)+C_{S,2}. 
\end{split}} 
\end{equation}
This is similar to the bound in \cite[(15)]{AvgZig}. Moreover, it is here that we use that $\infty$ and $0$ are not in the orbit of $P$ and $0$, in order to apply (\ref{ht:lbd}). Now we use properties of canonical heights. Specifically, Lemma \ref{lem:particularcanonicalheight} and Remark \ref{rem:ht} together imply that 
\begin{equation}\label{abc5}
\Big|h(\rho_r\cdot Q)-\deg(\rho_r)\hat{h}_\rho(Q)\Big|\leq C_S
\end{equation} 
for all $r\geq1$, all $\rho\in\Phi_S$, and all $Q\in\mathbb{P}^1(\overline{K})$; here $C_S$ is the height control constant from Theorem \ref{thm:htsiid}. In particular, by letting $\epsilon=1/2$ and using the fact that $\deg(F_n)\geq4$, we deduce from (\ref{abc4}) and (\ref{abc5}) that
\begin{equation}\label{abc6}
\begin{split} 
\deg(\gamma_{n-1})\hat{h}_\gamma(P)\leq&\; 2\bigg(\hat{h}_\gamma(P)\sum_{m=1}^{\lfloor\frac{n}{2}\rfloor}\deg(\gamma_m)+\sum_{m=\lfloor\frac{n}{2}\rfloor}^{n-1}\deg(\gamma_n\mysetminus\gamma_m)\hat{h}_{\gamma\mysetminus\gamma_m}(0)\bigg)\\ 
&\;+C_{S,3}n+C_{S,4}.
\end{split} 
\end{equation} 
Therefore, after dividing both sides of (\ref{abc6}) by $\deg(\gamma_{n-1})\hat{h}_\gamma(P)$ and noting that $\hat{h}_{\gamma\mysetminus\gamma_m}(0)\leq C_S$ for all $m$, we achieve a bound of the form \vspace{.25cm} 
\begin{equation}\label{abc7}
\Scale[0.86]{
\begin{split}
1\leq &\; \frac{C_{S,5,\gamma,P}\Big(\sum_{m=1}^{\lfloor\frac{n}{2}\rfloor}\frac{\deg(\gamma_m)}{\deg(\gamma_{\lfloor n/2\rfloor})}\Big)}{\deg\big(\gamma_{n-1}\mysetminus\gamma_{\lfloor n/2\rfloor}\big)}+\frac{\deg(\theta_n)\,C_{S,6,\gamma,P}\Big(\sum_{m=\lfloor\frac{n}{2}\rfloor}^{n-1}\frac{\deg(\gamma_n\mysetminus\gamma_m)}{\deg(\gamma_n\mysetminus\gamma_{\lfloor n/2\rfloor})}\Big)}{\deg(\gamma_{\lfloor n/2\rfloor})}+\frac{C_{S,7,\gamma,P}\,n}{\deg(\gamma_{n-1})}+\frac{C_{S,8,\gamma,P}}{\deg(\gamma_{n-1})} \\[11pt] 
\leq&\;\frac{C_{S,9,\gamma,P}}{\deg\big(\gamma_{n-1}\mysetminus\gamma_{\lfloor n/2\rfloor}\big)} + \frac{C_{S,10,\gamma,P}}{\deg\big(\gamma_{\lfloor n/2\rfloor}\big)} +\frac{C_{S,7,\gamma,P}\,n}{\deg(\gamma_{n-1})}+\frac{C_{S,8,\gamma,P}}{\deg(\gamma_{n-1})}.\\[8pt]   
\end{split} }
\end{equation}
However, the right-hand side of (\ref{abc7}) goes to zero as $n$ grows. Hence, $n$ is bounded and $\mathcal{Z}(\gamma,P)$ is finite as claimed.    
\end{proof} 
\begin{remark}\label{sqf} Conditions (1) and (2) of Theorem \ref{thm:primdiv} are equivalent to writing $\phi(x)=\frac{F_\phi(x)}{G_\phi(x)}$ with $\disc(F_\phi)\neq0$ and $\deg(F_\phi)\geq4$ for all $\phi\in S$.  
\end{remark} 
We now study dynamical Zsigmondy sets when $K/\mathbb{F}_p(t)$ is a finite separable extension and $p$ is odd. To do this, we translate the proof of \cite[Theorem 1.1]{Riccati} to the non-autonomous setting and use properties of canonical heights from Section \ref{sec:hts}. For completeness, we remind the reader of the definition of the Weil height in this setting. Given a valuation $v$ on $K$ with residue field $k_v$ and local degree $N_v$ \cite[Definition 1.1.14]{ffields}, we define the height of $\alpha\in K^*$ to be  
\begin{equation}\label{htff} 
h(\alpha)=-\sum_{v} \min\{v(\alpha),0\}N_v=\sum_{v}\max\{v(\alpha),0\} N_v,
\end{equation}  
where the equality above follows from the fact that on a curve, the number of zeros of a non-constant function equals the number of poles when counted with multiplicity; see, for instance, \cite[Theorem 1.4.11]{ffields}.  
 
Before we begin our proof of Theorem \ref{thm:primdivff}, we record a few auxiliary results, including Lemma 2.1 from \cite{Riccati} stated below. However, given the technical nature of the material that follows, the reader is encouraged to keep in mind that our overall strategy (as in the proof of \cite[Theorem 1.1]{Riccati}) is to show that for any sufficiently long string $\psi_m\in\Phi_{S,m}$ there exists a root of $\psi_m$ that fails to satisfy a Riccati equation. Here a \emph{Riccati equation} is a first order differential equation of the form
\[y'=ay^2+by+c.\]
In what follows, $K^{\sep}$ denotes the separable closure of $K$. 
\begin{lemma}{\label{lem:unique}} Let $K/\mathbb{F}_p(t)$, let $\phi(x)\in K[x]$ have degree $d\geq3$, and write \vspace{.025cm} 
\[\phi(x)=A_0x^d+A_{1}x^{d-1}+\dots A_{d-1}x+A_d.\vspace{.075cm}\]    
If $d\in K^*$ and the quantity \[\delta_\phi:=2dA_0A_{2}-(d-1)A_{1}^2\vspace{.15cm}\]
is non-zero, then for all $\beta\in K^{\sep}$ such that $\beta$ and $\phi(\beta)$ both satisfy a Riccati equation, i.e.\vspace{.1cm} 
\begin{equation}{\label{Ricatti}} \beta'=a\beta^2+b\beta+c\;\;\;\text{and}\;\;\;\phi(\beta)'=e\phi(\beta)^2+f\phi(\beta)+g\vspace{.1cm}
\end{equation} 
for some $a,b,c,e,f,g\in K$, either \[[K(\beta):K]\leq2d,\] or the coefficients in (\ref{Ricatti}) are uniquely determined by $\phi$:  \vspace{.15cm} 
\begin{equation}{\label{solutions}}
\Scale[0.95]{
\begin{split} 
&\,a=0,\;\;\;\; b=(dA_0^2A_2' - (d-1)A_0A_1A_1' - dA_0A_2A_0' + (d-1)A_1^2A_0')/\delta_\phi ,\\[8pt] 
&\,e=0,\;\;\; f=(d^2A_0^2A_2' - d(d-1)A_0A_1A_1' - d(d-2)A_0A_2A_0' + (d(d-2)+1)A_1^2A_0')/\delta_\phi, \\[8pt] 
&\,c=(A_0A_1A_2' - 2A_0A_2A_1' + A_1A_2A_0')/\delta_\phi, \;\;\; g=A_{d-1}c-A_df+A_d'. \\[3pt]  
\end{split}}    
\end{equation} 
\end{lemma} 
We gather another result that uses our technical assumptions from Theorem \ref{thm:primdivff}. In particular, we show that not all roots of sufficiently long strings of elements of $S$ can satisfy a Riccati equation over $K$. It is in this step especially where we encounter the subtlety of iterating multiple maps.    
\begin{lemma}\label{substring} Let $K/\mathbb{F}_p(t)$ be a function field and suppose that $S=\{\phi_1,\phi_2,\dots, \phi_s\}$ satisfies the hypothesis of Theorem \ref{thm:primdivff}. 
Then there exists a constant $M=M(S)$ such that for any string $\psi_M=(\theta_n)_{n=1}^M\in\Phi_{S,M}$ of length $M$ there is an integer $0\leq m_\psi\leq 2$ and a root $\beta$ of the substring $(\psi_M\mysetminus\psi_{m_\psi})=(\theta_n)_{n=m_\psi+1}^M$ that satisfies: 
\begin{enumerate}[topsep=8pt, partopsep=8pt, itemsep=10pt] 
\item[\textup{(a)}] $\beta\in K^{\sep}$, 
\item[\textup{(b)}] $[K(\beta):K]>2d$, 
\item[\textup{(c)}] $\beta$ does not satisfy a Riccati equation over $K$. 
\end{enumerate} 
\end{lemma} 
\begin{proof} To find the pair $m_\psi$ and $\beta$, let $d:=\max\{\deg(\phi): \phi\in S\}$ and note that \vspace{.05cm} 
\begin{equation*}{\label{htmin}} \hat{h}_{\Phi_S,K}^{\min}(2d):=\inf\Big\{\,\hat{h}_\psi(\alpha)\,:\; [K(\alpha):K]\leq2d,\;\alpha\not\in\PrePer(\Phi_S),\; \psi\in\Phi_S \Big\} \vspace{.05cm} 
\end{equation*} 
is a positive number; this follows from Lemma \ref{lem:htmin} below. Now let $m\geq1$ and suppose that $\alpha\in\mathbb{P}^1(\overline{K})$ is a root of a string $\psi_m\in\Phi_{S,m}$ and that $[K(\alpha):K]\leq2d$. Extend $\psi_m$ to some infinite sequence $\psi\in\Phi_S$, and note that Remark \ref{rem:ht} implies that \vspace{.05cm}
\begin{equation}\label{globalbd1} 
\hat{h}_{\psi\mysetminus\psi_m}(0)=\hat{h}_{\psi\mysetminus\psi_m}(\psi_m(\alpha))=\deg(\psi_m)\,\hat{h}_\psi(\alpha)\geq\deg(\psi_m)\,\hat{h}_{\Phi_S,K}^{\min}(2d). \vspace{.05cm}
\end{equation} 
Here we use our assumption that $0\not\in\PrePer(\Phi_S)$, so that $\alpha\not\in\PrePer(\Phi_S)$ also. Now let $C_S$ be the height control constant from Theorem \ref{thm:htsiid}, and note that Lemma \ref{lem:particularcanonicalheight} implies that $\hat{h}_{\psi\mysetminus\psi_m}(0)\leq C_S$, independent of the extension $\psi$ of $\psi_m$. Therefore, in light of (\ref{globalbd1}), we define the constant   
\begin{equation}\label{globalbd2}
\boxed{M:=\log_3^+\big(C_S\big/\hat{h}_{\Phi_S,K}^{\min}(2d)\big)+3},
\end{equation}  
where $\log_3^+(z)=\max\{\log_3(z),0\}$. To see that $M$ has the desired properties, let $\psi_M\in\Phi_{S,M}$ be any string. Note first that we may choose a separable root $\beta_M$ of $\psi_M$: for if every root of $\psi_M$ is inseparable, then $\deg(\psi_M)$ is divisible by $\car(K)$; see \cite[\S13.5 Proposition 38]{DF}. However, this contradicts assumption (1) of Theorem \ref{thm:primdivff}. Now write the string $\psi_M=(\theta_n)_{n=1}^M$ explicitly. Then we claim that $(m_\psi,\beta)$ in Lemma \ref{substring} can be chosen to be one of the pairs below: 
\[(m_\psi,\beta)=\big(0,\, \beta_M\big),\hspace{.2cm} \big(1,\,\theta_{1}(\beta_M)\big)\hspace{.15cm} \text{or}\hspace{.15cm} \big(2,\,\theta_{2}(\theta_{1}(\beta_M))\big).\vspace{.05cm}\]    
To see this, note that condition (a) of Lemma \ref{substring} is easily satisfied: since $\beta_M$ is separable, both $\theta_{1}(\beta_M)$ and $\theta_{2}(\theta_{1}(\beta_M))$ are also separable. On the other hand, the relative degrees $\big[K(\beta_M):K\big]$, $\big[K(\theta_{1}(\beta_M)):K\big]$ and $\big[K(\theta_{2}(\theta_{1}(\beta_M))):K\big]$ are all at least $2d$ by (\ref{globalbd1}) and the definition of $M$ in (\ref{globalbd2}). Therefore, it remains to show that at least one of the algebraic functions $\beta$, $\theta_1(\beta_M)$ or $\theta_2(\theta_1(\beta_M))$ does not satisfy a Riccati equation over $K$. Suppose, for a contradiction, that all three satisfy a Riccati equation over $K$. Then Lemma \ref{lem:unique} applied separately to the point $\beta=\beta_M$ with the map $\phi=\theta_1$ and the point $\beta=\theta_1(\beta_M)$ with the map $\phi=\theta_2$ implies that \vspace{.05cm}
\begin{equation} 
\begin{split} 
\beta'_M=b_1\beta_M+c_1\;\;\;\;&\text{and}\;\;\;\;\theta_1(\beta_M)'=f_1\theta_1(\beta_M)+g_1,\\[3pt]
\theta_1(\beta_M)'=b_2\theta_1(\beta_M)+c_2\;\;\;\;&\text{and}\;\;\;\;\theta_2(\theta_1(\beta))'=f_2\theta_2(\theta_1(\beta))+g_2\\[3pt] 
\end{split} 
\end{equation}  
for the unique solutions $(b_1,c_1,f_1,g_1)\in K^4$ and $(b_2,c_2,f_2,g_2)\in K^4$ in (\ref{solutions}) corresponding to $\phi=\theta_1$ and $\phi=\theta_2$ respectively. In particular, $0=(b_2-f_1)\,\theta_1(\beta_M)+(c_2-g_1)$. But $[K(\theta_1(\beta_M)):K]>2d$, so that $(b_2-f_1)=0$. However, this again contradicts condition (1) of Theorem \ref{thm:primdivff}.               
\end{proof}  
The importance of Riccati equations to the arithmetic of function fields is partially explained by their ability to detect the isotriviality of hyperelliptic curves \cite[Lemma 2.2]{Riccati}.  
\begin{lemma}{\label{lem:RicattitoIso}} Let $K/\mathbb{F}_p(t)$ and suppose that $\rho(x)\in K[x]$ is an irreducible (and separable) polynomial of degree $d\geq5$. If $\beta\in K^{\sep}$ is such that $\rho(\beta)=0$ and $\beta$ does not satisfy a Riccati equation over $K$, i.e. 
\[\beta'\neq a\beta^2+b\beta+c\]
for all $a,b,c\in K$, then the hyperelliptic curve\, $C: Y^2=\rho(X)$ is non-isotrivial.   
\end{lemma}  
As a final preparation for the proof of Theorem \ref{thm:primdivff}, we replace Proposition \ref{prop:abc} with effective forms of the Mordell conjecture over global function fields \cite{Kim,htineq,Szpiro}. 
\begin{theorem}\label{thm:Mordell}{(Effective Mordell Conjecture)} Let $X$ be a non-isotrivial curve of genus at least $2$ defined over a function field $K$ of characteristic $p>0$. Then there are constants $B_{X,1}$ and $B_{X,2}$ depending on $X$ such that for all $Q\in X(\overline{K})$, 
\[h_{\kappa_X}(Q)\leq B_{X,1}\,d(Q)+ B_{X,2};\]
here $h_{\kappa_X}$ is a height function with respect to the canonical divisor $\kappa_X$ of $X$ and 
\[d(Q)=\frac{2\gen(K(Q))-2}{[K(Q):K]}\] 
is the relative discriminant of the extension $K(Q)/K$.  
\end{theorem} 
\begin{proof} Strictly speaking, the bounds in \cite{Kim,htineq,Szpiro} are for curves with non-zero Kodaira-Spencer class. However, the non-isotrivial case follows from this one; see \cite[Theorem 5]{Kim-notes} or \cite[Theorem 2]{Voloch-Survey}. For if $X$ is a non-isotrivial curve, then there is an $r$ (a power of $p$) and a separable extension $L/K$ such that $X$ is defined over $L^r$ and that the Kodaira-Spencer class of $X$ over $L^r$ is non-zero \cite[Appendix]{Voloch-Survey}. Now, if we apply any of the bounds in \cite{Kim,htineq,Szpiro} to $X$ over $L^r$, then we achieve the bound in Theorem \ref{thm:Mordell}. 
\end{proof} 
Having completed the requisite preparations, we can now prove the finiteness of some polynomial Zsigmondy sets over global function fields. 
\begin{proof}[(Proof of Theorem \ref{thm:primdivff})] Let $\gamma\in\Phi_S$, let $P\in\mathbb{P}^1(K)$ be such that $\hat{h}_\gamma(P)>0$, and suppose that $n\in\mathcal{Z}(\gamma,P)$. Without loss of generality, we may assume that  $n>M$; see (\ref{globalbd2}) for a definition of $M$. Write $\gamma_n=(\theta_m)_{m=1}^n$ so that 
\[\gamma_n=\psi_M(\gamma)\circ\gamma_{n-M},\;\;\;\text{where}\;\;\;\psi_{M}(\gamma):=(\theta_m)_{m=n-M+1}^n\;\;\; \text{and}\;\;\;\gamma_{n-M}:=(\theta_m)_{m=1}^{n-M}.\] 
Now since $\psi_M(\gamma)$ is a string of length $M$, Lemma \ref{substring} implies that $\psi_M(\gamma)=\rho_M(\gamma)\circ\psi_{m_\psi}(\gamma)$, where $\psi_{m_\psi}$ is a string of length $0$, $1$ or $2$ and $\rho_M(\gamma)$ has a root $\beta$ satisfying conditions (1)-(3) of Lemma \ref{substring}. In particular, the polynomial $\rho_M(\gamma)$ has a factorization 
\[\rho_M(\gamma)=f_{\gamma,M,1}(x)^{e_1}f_{\gamma,M,2}(x)^{e_2}\dots f_{\gamma,M,r}(x)^{e_r},\]
satisfying: the $f_{\gamma,M,i}\in K[x]$ are distinct and irreducible, $\deg(f_{\gamma,M,1})\geq6$, and $f_{\gamma,M,1}$ has a separable root $\beta_{M,\gamma}$ that does not satisfy a Riccati equation over $K$. In particular, the hyperelliptic curve 
\begin{equation*}\label{specialcurve} \mathcal{C}_{\gamma,M}:\; Y^2=f_{\gamma,M,1}(X) 
\end{equation*}     
is non-isotrivial by Lemma \ref{lem:RicattitoIso}. Moreover, there are only finitely many such curves $\mathcal{C}_{\gamma,M}$ to consider since $\rho_M(\gamma)\in\Phi_{S, M-i}$ for some $0\leq i\leq2$. In particular, there are uniform height bound constants 
\[\mathcal{B}_{S,1}:=\max_{\gamma\in\Phi_S}\{B_{\mathcal{C}_{\gamma,M},1}\}\hspace{.5cm}\text{and}\hspace{.5cm} \mathcal{B}_{S,2}:=\max_{\gamma\in\Phi_S}\{B_{\mathcal{C}_{\gamma,M},2}\}\]
from Theorem \ref{thm:Mordell}. From here, the proof that $n$ is bounded is (mutatis mutandis) the same as the proof of \cite[Theorem 1.1]{Riccati} and similar to the proof of Theorem \ref{thm:primdiv} above. Specifically, we consider the algebraic point  
\[Q_{n,\gamma,P}=\Big((\psi_{m_\psi}(\gamma)\circ\gamma_{n-M})(P),\sqrt{(f_{\gamma,M,1}\circ(\psi_{m_\psi}(\gamma)\circ\gamma_{n-M}))(P)}\,\Big)\in \mathcal{C}_{\gamma,M}(\overline{K}),\] 
so that Theorem \ref{thm:Mordell} implies that 
\begin{equation*}\label{globalbd3} 
h_{\kappa_{\mathcal{C}_{\gamma,M}}}(Q_{n,\gamma,P})\leq \mathcal{B}_{S,1}\,d(Q_{n,\gamma,P})+\mathcal{B}_{S,2}. 
\end{equation*}  
On the other hand, let $x_{n,\gamma}(P)=X(Q_{n,\gamma,P})$ and $y_{n,\gamma}(P)=Y^2(Q_{n,\gamma,P})$ respectively. Then   
\begin{equation}\label{globalbd4} 
h(x_{n,\gamma}(P))\leq \mathcal{B}_{S,3}\bigg(\sum_{\hspace{.2cm}v(y_{n,\gamma}(P))>0}\hspace{-.6cm}N_v+\sum_{\hspace{.2cm}v(y_{n,\gamma}(P))<0}\hspace{-.4cm}N_v\hspace{.2cm}\bigg)+\mathcal{B}_{S,4};
\end{equation}
here we use \cite[Proposition 3.7.3]{ffields} to calculate $d(Q_{n,\gamma,P})$ and use \cite[Theorem III.10.2]{SilvA} to compare the two heights $h_{\kappa_{\mathcal{C}_{\gamma,M}}}(Q_{n,\gamma,P})$ and $h(x_{n,\gamma}(P))$. However, $S$ is a set of polynomials, so that $v(y_{n,\gamma}(P))<0$ can only occur when $P$, the coefficients of the polynomials in $S$, or the coefficients of $f_{\gamma,M,1}$ have negative valuations. Therefore, we deduce from (\ref{globalbd4}) that 
\begin{equation}\label{globalbd5} 
h(x_{n,\gamma}(P))\leq \mathcal{B}_{S,5}\underset{\hspace{-.4cm}v(y_{n,\gamma}(P))>0}{\sum_{v\in\mathfrak{s}_P}N_v}+\mathcal{B}_{S,6}\,h(P)+\mathcal{B}_{S,7};  
\end{equation}
here $\mathfrak{s}_P$ is the set valuations $v$ satisfying the following: \vspace{.1cm}  
\begin{equation} 
\begin{split} 
\hspace{2cm}&\text{$\bullet$ $v(P)\geq0$,}\\
\hspace{2cm}&\text{$\bullet$ every polynomial $\phi\in S$ has good reduction at $v$},\\ 
\hspace{2cm}&\text{$\bullet$ $f_{\gamma,M,i}$ has good reduction at $v$ or all $i$ and all $\gamma\in\Phi_S$}, \\[2pt] 
\end{split} 
\end{equation} 
Now we use our assumption that $n\in\mathcal{Z}(\gamma,P)$. In particular, if $v\in\mathfrak{s}_P$ and $v(y_{n,\gamma}(P))>0$, then $v(\gamma_n(P))>0$. Therefore, $v(\gamma_m(P))>0$ for some $1\leq m\leq n-1$ by the definition of the Zsigmondy set. However, as in (\ref{abc3}), $v(\gamma_n(P))>0$ and $v(\gamma_m(P))>0$ implies that $v((\gamma_n\mysetminus\gamma_m)(0))>0$.  In particular, we see that (\ref{htff}) and (\ref{globalbd5}) imply that 
\begin{equation}\label{globalbd6} 
\Scale[0.88]{
\begin{split} 
h(x_{n,\gamma}(P))&\leq\mathcal{B}_{S,5}\,\bigg(\underset{v(\gamma_m\cdot P)>0}{\sum_{m=1}^{\lfloor\frac{n}{2}\rfloor}}\hspace{-.3cm}N_v\;\;+\hspace{-.15cm}\underset{v((\gamma_n\mysetminus\gamma_m)\cdot 0)>0}{\sum_{m=\lfloor\frac{n}{2}\rfloor}^{n-1}}\hspace{-.6cm}N_v\;\;\;\bigg)+\mathcal{B}_{S,6}\,h(P)+\mathcal{B}_{S,7}\\[6pt]
&\leq\mathcal{B}_{S,5}\,\bigg(\sum_{m=1}^{\lfloor\frac{n}{2}\rfloor}h(\gamma_m\cdot P)+\sum_{m=\lfloor\frac{n}{2}\rfloor}^{n-1}\hspace{-.2cm}h((\gamma_n\mysetminus\gamma_m)\cdot 0)\;\bigg)+\mathcal{B}_{S,6}\,h(P)+\mathcal{B}_{S,7};
\end{split}} 
\end{equation} 
compare to (\ref{abc4}) over number fields. However, $\big|h(\rho_r\cdot Q)-\deg(\rho_r)\hat{h}_\rho(Q)\big|\leq C_S$ for all $r\geq1$, all $\rho\in\Phi_S$, and all $Q\in\mathbb{P}^1(\overline{K})$, where $C_S$ is the height control constant from Theorem \ref{thm:htsiid}; see Lemma \ref{lem:particularcanonicalheight}. Therefore,  
\begin{equation}\label{globalbd7}
\begin{split} 
\,\deg(\gamma_{n-M})\hat{h}_\gamma(P)\leq&\;\,\mathcal{B}_{S,8}\bigg(\hat{h}_\gamma(P)\sum_{m=1}^{\lfloor\frac{n}{2}\rfloor}\deg(\gamma_m)+\sum_{m=\lfloor\frac{n}{2}\rfloor}^{n-1}\deg(\gamma_n\mysetminus\gamma_m)\hat{h}_{\gamma\mysetminus\gamma_m}(0)\bigg)\\
&\;+\mathcal{B}_{S,9}\,\hat{h}_\gamma(P)+\mathcal{B}_{S,10}\,n+\mathcal{B}_{S,11};
\end{split} 
\end{equation} 
compare this to the estimate (\ref{abc6}) in the number field setting. Here we use also the trivial bound $\deg(\gamma_{n-M})\leq\deg((\psi_{m_\psi}(\gamma)\circ\gamma_{n-M}))$, in an attempt to make the expressions less cumbersome. 

Finally, $\hat{h}_{\gamma\mysetminus\gamma_m}(0)\leq C_S$ for all $m$ by Lemma \ref{lem:particularcanonicalheight}, and we deduce that $n$ is bounded as in (\ref{abc7}) over number fields. Namely, divide both sides of (\ref{globalbd7}) by $\deg(\gamma_n)\hat{h}_\gamma(P)$ and see that the right hand goes to zero as $n$ grows while the left hand side is bounded below by $1/d^M$, where $d=\max\{\deg(\phi)\,:\phi\in S\}$. This completes the proof of Theorem \ref{thm:primdivff}.   
\end{proof}
We conclude this section by noting that there is an absolute (positive) lower bound on the canonical height of all non-preperiodic points of bounded degree, a fact used to establish Lemma \ref{substring}. 
\begin{lemma}\label{lem:htmin} Let $K$ be a global field and let $S=\{\phi_1,\phi_2,\dots,\phi_s\}$ be a set of endomorphisms on $\mathbb{P}^1$ all defined over $K$ and of degree at least $2$. Then for all $D>0$, the quantity 
\[\hat{h}_{\Phi_S,K}^{\min}(D):=\inf\Big\{\,\hat{h}_\psi(\alpha)\,:\; [K(\alpha):K]\leq D,\;\alpha\not\in\PrePer(\Phi_S),\; \psi\in\Phi_S \Big\}\] 
is strictly positive. 
\end{lemma} 
\begin{proof} 
Choose any $Q_0\in\mathbb{P}^1$ such that $[K(Q_0):K]\leq D$ and $Q_0\not\in\PrePer(\Phi_S)$; this is possible by Northcott's Theorem and Lemma \ref{lem:particularcanonicalheight}. Then for any $\rho\in\Phi_S$, we have that 
\begin{equation}\label{estimate1}\Scale[0.9]{ 
\hat{h}_{\Phi_S,K}^{\min}(D)=\inf\Big\{\,\hat{h}_\psi(\alpha)\,:\;\hat{h}_\psi(\alpha)<\hat{h}_\rho(Q_0),\,\; [K(\alpha):K]\leq D,\;\alpha\not\in\PrePer(\Phi_S),\; \psi\in\Phi_S \Big\}}.
\end{equation} 
On the other hand, if $\hat{h}_\psi(\alpha)<\hat{h}_\rho(Q_0)$, then $h(\alpha)<\hat{h}_\rho(Q_0)+C_S$ by Lemma \ref{lem:particularcanonicalheight}. Moreover, the set of points
\begin{equation}\label{estimate2}\Scale[0.92]{ 
T_D:=\Big\{\alpha\in\mathbb{P}^1(\overline{K})\,:\,[K(\alpha):K]\leq D,\;h(\alpha)<\hat{h}_\rho(Q_0)+C_S,\;\alpha\not\in\PrePer(\Phi_S) \Big\}}
\end{equation}  
is finite by Northcott's theorem. In particular, since  
\[\hat{h}_{\Phi_S,K}^{\min}(D)\geq \min_{\alpha\in T_D}\inf_{\psi\in\Phi_S}\{\hat{h}_\psi(\alpha)\}\] 
by (\ref{estimate1}) and (\ref{estimate2}) above and $T_D$ is finite, it suffices to show that $\inf_{\psi\in\Phi_S}\{\hat{h}_\psi(\alpha)\}$ is strictly positive for any non-preperiodic $\alpha$ to prove that $\hat{h}_{\Phi_S,K}^{\min}(D)>0$: the minimum value of a finite set of positive numbers is positive. To do this, we note first that \vspace{.1cm} 
\[\hspace{3.5cm}\big|\hat{h}_Q(\gamma)-\hat{h}_Q(\psi)\big|\leq C_S\,\Delta(\gamma,\psi)\;\;\;\;\;\text{for all $Q\in\mathbb{P}^1(\overline{K})$ and all $\gamma,\psi\in\Phi_S$}, \vspace{.1cm}\] 
where $\Delta(\gamma,\psi)=2^{-\min\{n\,:\,\gamma_n\neq\psi_n\}}$ gives a metric on $\Phi_S$; this follows easily from Lemma $\ref{lem:particularcanonicalheight}$ and Remark \ref{rem:ht} above. In particular, for any fixed $\alpha$ the canonical height map, $\hat{h}_\alpha:\Phi_S\rightarrow\mathbb{R}_{\geq0}$ given by $\gamma\rightarrow\hat{h}_\gamma(\alpha)$, is continuous; in fact, $\hat{h}_\alpha$ is uniformly continuous. In particular, since $\Phi_S$ is compact ($S$ is finite), $\hat{h}_\alpha$ must attain its minimum value somewhere on $\Phi_S$. In particular, this minimum value must be positive when $\alpha\not\in\PrePer(\Phi_S)$ by definition.       
\end{proof} 
\begin{remark} In fact, for non-preperiodic basepoints, we have established uniform bounds on Zsigmondy sets: assuming the hypothesis of Theorem's \ref{thm:primdiv} and \ref{thm:primdivff}, there exists an integer $N=N(S,K)$ such that if $n\in\mathcal{Z}(\gamma,P)$ for some $(\gamma,P)\in G_S$ with $P\not\in\PrePer(\Phi_S)$, then $n\leq N$. This follows from (\ref{abc6}), (\ref{globalbd7}), and Lemma \ref{lem:htmin}. 
\end{remark} 
\section{Local Height Decomposition}{\label{sec:local}}
As with the usual canonical height (associated to a single endomorphism), the canonical and expected canonical heights defined in Theorem \ref{thm:htsiid} can be written as a sum of local heights. To do this, we follow the construction given in \cite[\S5.9]{SilvDyn}. In particular, we assume that $V=\mathbb{P}^1$ and that our generating set of self maps $S=\{\phi_1,\phi_2,\dots, \phi_s\}$ is finite. Moreover, since most of the results in this section are straightforward adaptations of Lemma \ref{lem:uniformbd}, Lemma \ref{lem:particularcanonicalheight}, and results in \cite[\S5.9]{SilvDyn}, we simply provide an outline here.  

First some notation. Let $M_K$ be a complete set of absolute values on $K$. To each place $|\cdot|_v\in M_K$, we define a norm $\lVert\cdot\rVert_v$ on $\mathbb{A}^2(K_v)$ given by 
\[\lVert(x,y)\rVert_v=\max\{|x|_v,|y|_v\},\]
where $K_v$ is the completion of $K$ at $|\cdot|_v$; let $n_v$ be its local degree \cite[\S3.1]{SilvDyn}. Now choose and fix a lift $\tilde{\phi}_i:\mathbb{A}^2\rightarrow\mathbb{A}^2$ over $K$ for each $\phi_i\in S$, and associate a ``lift" to a sequence $\gamma=(\theta_n)_{n\geq1}\in\Phi_S$ by $\tilde{\gamma}=(\tilde{\theta}_n)_{n\geq1}$. Then we define an associated \emph{Green function} to $\tilde{\gamma}$ and $|\cdot|_v$ by analogy with the canonical height: 
\begin{equation}{\label{Green}} 
\mathcal{G}_{v,\tilde{\gamma}}(x,y)=\lim_{n\rightarrow\infty}\frac{\log\,\lVert\tilde{\gamma}_n(x,y)\rVert_v}{\deg(\tilde{\gamma}_n)},\;\;\;\;\; (x,y)\in\mathbb{A}^2_*(K_v);
\end{equation} 
here $\mathbb{A}^2_*(K_v)=\{(x,y)\in\mathbb{A}^2(K_v)\,: x\neq0\;\;\text{or}\;\;y\neq0\}$.     
We collect some facts about these Green functions below; compare to \cite[Proposition 5.58 and Theorem 5.59]{SilvDyn}. 
\begin{proposition}{\label{prop:local1}} The functions $\mathcal{G}_{v,\tilde{\gamma}}$ are well defined and satisfy the following properties: 
\begin{enumerate}[topsep=8pt, partopsep=6pt, itemsep=10pt] 
\item[\textup{(1)}] Let $(x,y)\in\mathbb{A}^2_*(K_v)$. Then \[\mathcal{G}_{v,\tilde{\gamma}}(x,y)=\log\,\lVert(x,y)\rVert_v+O_v(1)\\[3pt]\]  
for all $v\in M_K$. Moreover, $O_v(1)=0$ for all but finitely many $v\in M_K$.    
\item[\textup{(2)}] Let $(x,y)\in\mathbb{A}^2_*(K)$ be any representative of $P=[x,y]\in\mathbb{P}^1(K)$. Then 
\[\hat{h}_\gamma(P)=\sum_{v\in M_K} n_v\, \mathcal{G}_{v,\tilde{\gamma}}(x,y)\] 
for any choice of lifts (over $K$) of the elements of $S$.   
\end{enumerate} 
\end{proposition} 
\begin{proof} The argument that the limit defining $\mathcal{G}_{v,\tilde{\gamma}}$ exists is, mutatis mutandis, the same as that given to establish that $\hat{h}_\gamma$ is well defined: one uses Tate's telescoping argument plus the fact that $S$ is finite (and thus height controlled). In particular, the key fact in this setting is that there exist positive constants $c_{i,v}$ such that   
\begin{equation}{\label{localbd}}
\bigg|\log\,\lVert\tilde{\phi}_i(x,y)\rVert_v-\deg(\phi_i)\log\,\lVert(x,y)\rVert_v \bigg|\leq c_{i,v}\;\;\;\; \text{for all $(x,y)\in\mathbb{A}^2_*(K_v)$};
\end{equation} 
see \cite[Proposition 5.57(a)]{SilvDyn}. Hence, (\ref{localbd}) and the adapted version of Tate's telescoping argument given in (\ref{Tate}) and (\ref{particularcanonicalheight}) imply that \vspace{.1cm}  
\begin{equation}{\label{localbd2}}
\Bigg|\,\frac{\log\,\lVert\tilde{\gamma}_n(x,y)\rVert_v}{\deg(\tilde{\gamma}_n)}-\frac{\log\,\lVert\tilde{\gamma}_m(x,y)\rVert_v}{\deg(\tilde{\gamma}_m)}\,\Bigg|\leq \frac{\max_i\{c_{i,v}\}}{\alpha^m(\alpha-1)},\;\;\;\;\; \text{$0<m<n$;} \vspace{.05cm}
\end{equation} 
here $\alpha=\min\{\deg(\phi_i)\}\geq2$. Hence, the limit defining $\mathcal{G}_{v,\tilde{\gamma}}(x,y)$ is Cauchy and therefore converges. Moreover, letting $m=0$ and $n\rightarrow\infty$ establishes statement (1). On the other hand, if $v\in M_K$ satisfies:
\begin{equation}{\label{goodprimes}} 
\text{$v$ is nonarchimedean, $\lVert \tilde{\phi}_i\rVert_v=1$, and $|\Res(\tilde{\phi}_i)|_v=1$},
\end{equation} 
then $c_{i,v}=0$; see \cite[Proposition 5.57(b)]{SilvDyn}. Hence, for all but finitely $v\in M_K$, we have that $\mathcal{G}_{v,\tilde{\gamma}}(x,y)=\log\,\lVert(x,y)\rVert_v$ for all $(x,y)\in\mathbb{A}^2_*(K_v)$; to see this, note that (\ref{goodprimes}) holds for all $i$ for all but finitely many places (recall that we fix the lifts $\tilde{\phi}_i$ at the outset). This completes the proof of statement (1).  As for statement (2), we note first that if $c\in K_v^*$, then 
\begin{equation} \label{scale}
\mathcal{G}_{v,\tilde{\gamma}}(cx,cy)=\mathcal{G}_{v,\tilde{\gamma}}(x,y)+\log|c|_v;
\end{equation} 
the proof is identical to \cite[(5.38) pp.289]{SilvDyn}. In particular, if $c\in K^*$, then 
\begin{equation} \label{welldefined}
\;\;\;\;\sum_{v\in M_K} n_v\mathcal{G}_{v,\tilde{\gamma}}(cx,cy)=\sum_{v\in M_k}n_v\mathcal{G}_{v,\tilde{\gamma}}(x,y)\;\;\;\text{for all \,$(x,y)\in\mathbb{A}^2_*(K)$},
\end{equation} 
by the product formula \cite[Proposition 3.3]{SilvDyn}. Hence, the right hand side of statement (2) of Proposition \ref{prop:local1} is independent of the choice of representative for $P=[x,y]$. Now fix $P\in\mathbb{P}^1(K)$ and a representative $(x,y)$, and note that the sum in (\ref{welldefined}) is a finite sum, since $\mathcal{G}_{v,\tilde{\gamma}}(x,y)=\log\,\lVert(x,y)\rVert_v=0$ for all but finitely many $v\in M_K$; of course this time the finitely many primes depend on $(x,y)$. In fact, the stronger statement holds: there exists a finite set $T_{\tilde{\gamma},(x,y)}\subset M_K$, depending on $\tilde{\gamma}$ and $(x,y)$, such that for all $|\cdot|_v\not\in T_{\tilde{\gamma},(x,y)}$,  
\[\mathcal{G}_{n,v,\tilde{\gamma}}(x,y):=\frac{\log\,\lVert\tilde{\gamma}_n(x,y)\rVert_v}{\deg(\tilde{\gamma}_n)}=\log\,\lVert(x,y)\rVert_v=0\;\;\;\text{for all $n\geq1$};\]
see (\ref{localbd}) and (\ref{goodprimes}) above. Hence, 
\begin{equation*}
\begin{split} 
\sum_{v\in M_K} n_v\, \mathcal{G}_{v,\tilde{\gamma}}(x,y)&=\sum_{\mathclap{v\in T_{\tilde{\gamma},(x,y)}}} n_v\, \mathcal{G}_{v,\tilde{\gamma}}(x,y)=\sum_{\mathclap{v\in T_{\tilde{\gamma},(x,y)}}} n_v\, \lim_{n\rightarrow\infty}\mathcal{G}_{n,v,\tilde{\gamma}}(x,y)\\[6pt]
&=\lim_{n\rightarrow\infty}\,\sum_{\mathclap{v\in T_{\tilde{\gamma},(x,y)}}} n_v\,\mathcal{G}_{n,v,\tilde{\gamma}}(x,y)=\lim_{n\rightarrow\infty}\,\sum_{v\in M_K} n_v \mathcal{G}_{n,v,\tilde{\gamma}}(x,y)\\[6pt]
&= \lim_{n\rightarrow\infty}\frac{\sum_{v\in M_K}n_v\log\,\lVert\tilde{\gamma}_n(x,y)\rVert_v}{\deg(\tilde{\gamma}_n)}=\lim_{n\rightarrow\infty}\frac{h(\gamma_n\cdot P)}{\deg(\tilde{\gamma}_n)}\\[2pt]
&=\hat{h}_\gamma(P)
\end{split} 
\end{equation*} 
as claimed. This completes the proof of Proposition \ref{prop:local1}.   
\end{proof}  
\begin{remark} We note that $\mathcal{G}_{v,\tilde{\gamma}}: \mathbb{A}^2_*(K)\rightarrow\mathbb{R}$ is continuous, since the sequence of continuous functions $\mathcal{G}_{n,v,\tilde{\gamma}}$ converges uniformly to $\mathcal{G}_{v,\tilde{\gamma}}$ by (\ref{localbd2});  compare to \cite[Proposition 5.58(e)]{SilvDyn}.     
\end{remark} 
Following \cite[\S5.9]{SilvDyn}, we use Green functions to define local canonical heights, which in some sense measure the the $v$-adic distance from points to divisors. Moreover, these functions are defined on Zariski open subsets of $\mathbb{P}^1$ (the ambient space), unlike the Green functions.   

To do this, let $E\in K[x,y]$ be a homogenous polynomial of degree $\deg(E)=e$ (which determines a divisor of $\mathbb{P}^1$). For a lift $\tilde{\gamma}$ of $\gamma\in \Phi_S$, determined by fixed lifts of the elements of $S$, we define the \emph{local canonical height} at $v$ associated to the pair $(\tilde{\gamma},E)$ to be the function:  
\begin{equation}\label{localht} 
\hat{\lambda}_{v,\tilde{\gamma}, E}([x,y]):= e \mathcal{G}_{v,\tilde{\gamma}}(x,y)-\log\, |E(x,y)|_v 
\end{equation}   
for all $[x,y]\in \mathbb{P}^1(K_v)$ with $E(x,y)\neq0$. We collect some properties of these local canonical height functions below; compare to \cite[Theorems 5.60, 5.61]{SilvDyn}. 
\begin{theorem}\label{thm:localhts} Let $E\in K[x,y]$ be a homogenous polynomial of degree $e$ and let $\tilde{\gamma}$ be a lift of $\gamma\in \Phi_S$. Then the following statements hold:  
\begin{enumerate}[topsep=8pt, partopsep=6pt, itemsep=10pt] 
\item[\textup{(1)}] $\hat{\lambda}_{v,\tilde{\gamma}, E}: \mathbb{P}^1(K_v)\mysetminus\{E=0\}\rightarrow\mathbb{R}$ is a well defined function.   
\item[\textup{(2)}] The function $P\rightarrow\hat{\lambda}_{v,\tilde{\gamma}, E}(P)+\log\frac{|E(P)|_v}{\lVert P\rVert_v^e}$ extends to a bounded continuous function on all of $\mathbb{P}^1(K_v)$.
\item[\textup{(3)}] The canonical height has a decomposition as a sum of local canonical heights: 
\[\hat{h}_\gamma(P)=\frac{1}{\deg(E)}\,\sum_{\mathclap{v\in M_K}}n_v\,\hat{\lambda}_{v,\tilde{\gamma}, E}(P)\]
for all $P\in\mathbb{P}^1(K)\mysetminus\{E=0\}$. 
\end{enumerate} 
\end{theorem} 
\begin{proof} The proof of statement (1) follows directly from (\ref{scale}). As for statement (2), note that 
\[ \hat{\lambda}_{v,\tilde{\gamma}, E}(P)+\log\frac{|E(P)|_v}{\lVert P\rVert_v^e}=e\big(\mathcal{G}_{v,\tilde{\gamma}}(P) - \log\,\lVert P\rVert_v\big)\]
is bounded by Proposition \ref{prop:local1} part (1). Moreover, both $\mathcal{G}_{v,\tilde{\gamma}}$ and $\log\,\lVert \cdot\rVert_v$ are continuous functions on $\mathbb{A}_*^2(K_v)$, and hence their difference is also continuous on $\mathbb{A}_*^2(K_v)$. Furthermore, since this difference is invariant under scaling, the map $P\rightarrow\hat{\lambda}_{v,\tilde{\gamma}, E}(P)+\log\frac{|E(P)|_v}{\lVert P\rVert_v^e}$ is continuous as claimed. Finally, if $P=[x,y]\in\mathbb{P}^1(K)\mysetminus\{E=0\}$, then \vspace{.2cm} 
\begin{equation}\label{localdecomp}
\begin{split} 
\,\hat{h}_\gamma(P)=\sum_{v\in M_k}n_vG_{v,\tilde{\gamma}}(x,y)&=\sum_{v\in M_k}\frac{n_v}{\deg(E)}\Big(\hat{\lambda}_{v,\tilde{\gamma}, E}([x,y]) +\log\, |E(x,y)|_v \Big)\\[8pt]
&=\frac{1}{\deg(E)}\sum_{\mathclap{v\in M_K}}n_v\,\hat{\lambda}_{v,\tilde{\gamma}, E}(P)\; + \frac{1}{\deg(E)}\sum_{\mathclap{v\in M_K}}n_v\, \log\, |E(x,y)|_v \\[8pt] 
&= \frac{1}{\deg(E)}\sum_{\mathclap{v\in M_K}}n_v\,\hat{\lambda}_{v,\tilde{\gamma}, E}(P)
\end{split}
\end{equation}
as claimed; here, we use (\ref{localht}), Proposition \ref{prop:local1} part (2), and the product formula, which implies that $\sum_{v\in M_K}n_v\log\, |E(x,y)|_v$ vanishes.        
\end{proof} 
We now fix the point $P\in\mathbb{P}^1(K)$ and vary its orbit in $\mathbb{P}^1$ by varying $\gamma\in\Phi_S$. In particular, we consider the functions $\gamma\rightarrow\hat{\lambda}_{v,E,P}(\gamma):=\hat{\lambda}_{v,\tilde{\gamma}, E}(P)$ for each $v\in M_K$. As a first observation, note that we may interpret Theorem \ref{thm:localhts} part (3) as a way of writing the random variable $\gamma\rightarrow \hat{h}_P(\gamma)$ as a sum of (local) random variables:   
\[\hat{h}_P(\gamma)=\sum_{v\in M_K}n_v\,\hat{\lambda}_{v,E,P}(\gamma).\]  
Since it is often a useful technique in probability theory to decompose a complicated random variable into a sum of independent random variables, we ask the following question: 
\begin{question} For what points $P\in \mathbb{P}^1(K)$, is $\Big\{\hat{\lambda}_{v,E,P}(\gamma)\Big\}_{v\in M_K}$a collection of independent random variables on $\Phi_S$? 
\end{question} 
Finally, we note that $\mathbb{E}_\nu[\hat{h}]:\mathbb{P}^1\rightarrow\mathbb{R}$, a height function that in some sense packages together the collective dynamics of the functions in $S$ at the point $P$, also has a decomposition into a sum of local pieces. 
\begin{corollary}{\label{cor:localhts}} Let $K$ be a global field, let $E\in K[x,y]$ be a homogenous polynomial, and let $S=\{\phi_1,\phi_2,\dots,\phi_s\}$ be a finite set of endomorphisms (over $K$) on $\mathbb{P}^1$, all of degree at least $2$. Fix lifts $\tilde{\phi}_i:\mathbb{A}^2\rightarrow\mathbb{A}^2$ for each $\phi_i$, and extend these to lifts of each element in $\Phi_S$. Then 
\begin{enumerate}[topsep=8pt, partopsep=6pt, itemsep=10pt] 
\item[\textup{(1)}] For all $v\in M_K$, the local expected canonical height, 
\[\mathbb{E}_\nu\big[\hat{\lambda}_{v,\tilde{\gamma},E}\big](P):=\int_{\Phi_S}\hat{\lambda}_{v,E,P}(\gamma)\,d\nu,\] 
is well defined for all $P\in\mathbb{P}^1(K_v)\mysetminus\{E=0\}$.   
\item[\textup{(2)}] The expected canonical height has a decomposition as a sum local expected canonical heights: 
\[ \mathbb{E}_\nu\big[\hat{h}\big](P)=\frac{1}{\deg(E)}\sum_{v\in M_K}n_v\,\mathbb{E}_\nu\big[\hat{\lambda}_{v,\tilde{\gamma},E}\big](P)\]
for all $P\in \mathbb{P}^1(K)\mysetminus\{E=0\}$.
\end{enumerate}    
\end{corollary} 
\begin{proof} The first statement follows from the Lebegue dominated convergence theorem: for each $n$ define the random variables $\lambda_{v,E,P,n}: \Phi_S\rightarrow \mathbb{R}$ given by  
\[\lambda_{v,E,P,n}(\gamma)=\deg(E)\,\frac{\log\lVert\tilde{\gamma}_n(x,y) \rVert_v}{\deg(\tilde{\gamma})}-\log|E(x,y)|_v;\]
here $(x,y)\in\mathbb{A}_*^2(K_v)$ is any representative of $P$. Then (\ref{localbd2}) implies that $\lambda_{v,E,P,n}$ is a sequence of uniformly bounded functions (take $m=0$). Furthermore, $\hat{\lambda}_{v,E,P}$ is the pointwise limit of the $\lambda_{v,E,P,n}$. Hence, \cite[Theorem 9.1]{ProbabilityText} implies that $\hat{\lambda}_{v,E,P}$ is integrable and  
\[ \mathbb{E}_\nu\big[\hat{\lambda}_{v,\tilde{\gamma},E}\big](P):=\int_{\Phi_S}\hat{\lambda}_{v,E,P}(\gamma)\,d\nu=\lim_{n\rightarrow\infty}\int_{\Phi_S}\lambda_{v,E,P,n}(\gamma)\,d\nu.\]
Moreover, the decomposition in statement (2) follows directly from Theorem \ref{thm:localhts} part (3) and the linearity of the integral: for any fixed $P$, the sum is in fact a finite sum; see (\ref{goodprimes}).    
\end{proof} 
 
\end{document}